\documentclass[final,5p,times,twocolumn]{elsarticle}

\usepackage{amssymb}
\usepackage{amsmath}
\usepackage{amsthm}

\usepackage{algorithm}
\usepackage{algorithmicx}
\usepackage{xcolor}
\usepackage[noend]{algpseudocode}
\algrenewcommand\algorithmicrequire{\textbf{Input:}}
\algrenewcommand\algorithmicensure{\textbf{Output:}}

\journal{}

\newtheorem{theorem}{Theorem}
\numberwithin{theorem}{section}
\newtheorem{lemma}[theorem]{Lemma}
\newtheorem{proposition}[theorem]{Proposition}
\newtheorem{remark}[theorem]{Remark}

\newtheorem{assumption}[theorem]{Assumption}

\newtheorem{problem}[theorem]{Problem}

\newcommand{\rv}[1]{\boldsymbol{#1}}                      
\newcommand{\drm}{\mathrm{d}}                             
\newcommand{\sgnMeasure}[1]{{\mathcal{M}(#1)}}            
\newcommand{\pdMeasure}[1]{{\mathcal{M}_+(#1)}}           
\newcommand{\boundedfunctional}[1]{{\mathcal{F}(#1)}}     
\newcommand{\contfnc}[1]{{\mathcal{C}(#1)}}               
\newcommand{\mR}{\mathbb{R}}                              
\newcommand{\scpair}{\eta}                               
\newcommand{\scspace}{{X\times A}}                        
\newcommand{\trsknl}{Q}                                   
\newcommand{\optm}[1]{#1^\star}                           
\newcommand{\adj}[1]{{#1^*}}                              
\newcommand{\expectation}[1]{{\mathbb{E}\left[{#1}\right]}}
\newcommand{\proj}{\mathcal{P}}                           
\newcommand{\initdistr}{\bar{\mu}_{\mathrm{init}}}        

\DeclareMathOperator*{\plim}{plim}
\DeclareMathOperator*{\plimsup}{plimsup}
\DeclareMathOperator*{\partialorder}{\prec_{grlex}}
\DeclareMathOperator*{\argmin}{argmin}
\DeclareMathOperator*{\cov}{cov}
\DeclareMathOperator{\mE}{\mathbb{E}}

\begin{document}

\begin{frontmatter}


\title{Consistent inverse optimal control for infinite time-horizon  discounted nonlinear systems under noisy observations\tnoteref{label1}}
\tnotetext[label1]{This work was supported by National Natural Science Foundation (NNSF) of China under Grant 62573286, by Natural Science Foundation of Shanghai under Grant 25ZR1401208, by the Wallenberg AI, Autonomous Systems and Software Program (WASP) funded by the Knut and Alice Wallenberg Foundation, and the Swedish Research Council (VR) under grant 2024-05776.}


\author[SJTU,IMR]{Ziliang Wang} 
\author[Chalmers]{Axel Ringh}
\author[SJTU,IMR]{Han Zhang\corref{cor1}}
\ead{zhanghan_tc@sjtu.edu.cn}
\cortext[cor1]{Corresponding author}

\affiliation[SJTU]{organization={School of Automation and Intelligent Sensing, Shanghai Jiao Tong University},
            city={Shanghai},
            country={China}}

\affiliation[IMR]{organization={Insitute of Medical Robotics, Shanghai Jiao Tong University},
            city={Shanghai},
            country={China}}

\affiliation[Chalmers]{organization={Department of Mathematical Sciences, Chalmers University of Technology and University of Gothenburg},
            city={Gothenburg},
            country={Sweden}}

\begin{abstract}
Inverse optimal control (IOC) aims to estimate the underlying cost that governs the observed behavior of an expert system. 
However, in practical scenarios, the collected data is often corrupted by noise, 
which poses significant challenges for accurate cost function recovery. 
In this work, we propose an IOC framework that effectively addresses the presence of observation noise. 
In particular, compared to our previous work \cite{wang2025consistent},
 we consider the case of discrete-time, infinite-horizon, discounted MDPs whose transition kernel is only weak Feller.
By leveraging the occupation measure framework, 
we first establish the necessary and sufficient optimality conditions for the expert policy 
and then construct an infinite dimensional optimization problem based on these conditions. 
This problem is then approximated by polynomials to get a finite-dimensional numerically solvable one, 
which relies on the moments of the state-action trajectory's occupation measure.
More specifically, the moments are robustly estimated from the noisy observations 
by a combined misspecified Generalized Method of Moments (GMM) estimator derived from observation model and system dynamics.
Consequently, the entire algorithm is based on convex optimization which alleviates the issues that arise from local minima
and is asymptotically and statistically consistent.
Finally, the performance of the proposed method is illustrated through numerical examples.
\end{abstract}




\end{frontmatter}



\section{Introduction}
The control behaviour of optimal control and Reinforcement Learning (RL) is determined by the cost (reward) function.
And one of the fundamental challenges in optimal control and RL is to design appropriate cost (reward) functions so that it can guide the autonomous agents to accomplish desired tasks or achieve a good control performance.
Inverse optimal control (IOC) seeks to infer the underlying cost function that explains the observed behavior of an expert system. When the expert's decision mechanism is modeled as solving an optimal control problem, IOC provides a principled mechanism for learning the underlying objectives from demonstrations and has found broad applications in robotics \citep{mombaur2010human,likmeta2021dealing}. In particular, after the underlying cost function that best suits the contextual environment is recovered, one can then deploy the learned cost function to synthesize policies for autonomous agents to perform similar tasks and achieves a similar performance to the expert.

Despite the conceptual appeal, accurate cost recovery remains challenging.
This is because in practical scenarios, demonstration data are rarely noise-free. Observation noise that arises from sensing limitations, state estimation, or data acquisition pipelines would inevitably contaminate the observed demonstration trajectories. Such noise fundamentally alters the statistical properties of the state-action pairs' distributions in the observed demonstration trajectory.
As shown by \citep{aswani2018inverse}, directly applying the method of Karush-Kuhn-Tucker (KKT) condition violation minimization to noisy data would lead to inconsistent cost estimates.
Consequently, in the IOC algorithm design, one needs to be careful when handling observation noise.
On the other hand, for general nonlinear stochastic optimal control systems, its value function that characterizes the expert's optimality is usually implicit, i.e., a closed-form expressions is not available. Hence the corresponding IOC problem does not have an explicit characterization of the optimality structure of the expert's behavior, which make the IOC algorithm design even more challenging.
Moreover, for IOC algorithms, it is desirable that they are free from local minimum issues. This is because it would not only lead to suboptimal cost estimates, but also make the consistency property of the designed IOC algorithm hard to guarantee in practice since the analysis is usually with repect to the global minimum.

\subsection{Related works}
As a special case, the IOC problem for linear quadratic (LQ) systems has been extensively studied in the literature.
In particular, \citep{li2018convex,Menner2018Convex} have been focused on the noiseless observation case, while \citep{guo2023imitation,yu2021system,zhang2019inverse,zhang2019inverseCDC,zhang2021inverse,zhang2022statistically,zhang2024statistically,zhang2024inverse,li2026inverse} consider the noisy observation case. Nevertheless, all these works are limited to LQ systems, which restricts their applicability to more general nonlinear systems.
For general nonlinear systems, \citep{rickenbach2024inverse} construct the IOC algorithm based Pontryagin's Minimum Principle (PMP) and see the problem as a error-in-variables problem. Nevertheless, the methods in \citep{rickenbach2024inverse} are only applicable for finite-time horizon deterministic systems, and they also rely on local optimization techniques, which are prone to local minimum issues and lack consistency theoratical guarantees.
Similar methods \citep{keshavarz2011imputing,molloy2018finite,molloy2020online,jin2019inverse} also minimize the PMP condition violation, and they also suffer from the same issues.
Furthermore, bilinear systems that results from Koopman lifting \citep{fernandez2025estimating}, control-affine systems \citep{lian2022inverse,Rodrigues2022Inverse} have also been considered, yet they all require a quadratic cost structure and do not consider noisy observations.
On the other hand, \citep{pauwels2016linear, rouot2017inverse} consider the IOC problem for general continuous-time finite horizon nonlinear systems based on the occupation measure \citep{hernandez2012discrete}, which leads to a polynomial optimization problem. However, these works do not consider noisy observations either. The asymptotic consistency regarding the polynomial approximation is also not established therein.
Another staightforward way to estimate the cost function is to minimize the differences between the observed policy and the optimal policy induced by the estimated cost function \citep{abbeel2004apprenticeship,ho2016generative}. Such methods usually require solving the forward optimal control problem repeatedly, which is computationally expensive. 
In addition, maximum entropy based IOC methods \citep{ziebart2008maximum,zhou2018Infinite,Mehr2023Maximum,garrabe2025convex} have been extensively studied in the literature, and they recover the cost function by maximumizing the likelihood of Boltzmann distribution that has the maximum entropy. Yet such strong assumption on the behaviour distribution is not always valid in practice, and it also requires solving the forward optimal control problem repeatedly.
\subsection{The contributions}
To address the above challenges, we propose an IOC framework that effectively handles noisy observations for discrete-time, infinite-horizon, discounted Markov Decision Processes (MDPs) whose transition kernel is only weak Feller.
In particular, the contribution is three-fold:
\begin{itemize}
  \item Compared to our prior work \citep{wang2025consistent}, which assumes the transition kernel is strongly Feller, this work relaxes such assumption. This means that the proposed IOC framework can now be applied to a broader class of systems, including those with deterministic dynamics.
  \item To deal with the observation noise, we propose to robustly estimate the moments of the state-action trajectory's occupation measure from noisy observations by a combined misspecified Generalized Method of Moments (GMM) estimator derived from observation model and system dynamics. This allows us to effectively mitigate the impact of observation noise on cost recovery.
  \item The entire algorithm is based on convex optimization, which spares from local minimum issues. Moreover, we rigorously prove that the proposed IOC algorithm is asymptotically and statistically consistent.
\end{itemize}

\textit{Notations:} For a set $G$, we denote $\boundedfunctional{G}$ as the set of bounded measurable functions, 
we also denote $\mathcal{B}(G)$ as the Borel-$\sigma$ algebra on the set $G$, and $\mathcal{M}_+(G)$, $\mathcal{M}(G)$ denotes the set of all finite non-negative, finite signed Borel measures on $G$, respectively.
Moreover, $\|v\|_W^2$ denote $v^\top W v$ for any vector $v$ and positive definite matrix $W$.
In addition, we use zero-based indexing and adopt the slicing notation $(\cdot)_{[i:j]}$ 
to denote sub-vectors and sub-matrices (e.g., $v_{[i:j]}$ extracts elements from index $i$ to $j-1$ of vector $v$), 
following the convention in \cite{harris2020array}.
For a matrix $A$, we denote $\sigma_{\min}(A)$ and $\sigma_{\max}(A)$ as the smallest and largest singular values of $A$, respectively.
We use \textbf{\textit{italic bold font}} to denote stochastic variables.
Furthermore, denote $\plim_{k\rightarrow\infty} \rv{s}_k = s$ as the limit of $\rv{s}_k$ and $\plimsup_{k\rightarrow\infty} \rv{s}_k$ as the limit superior of $\rv{s}_k$ as $k\rightarrow\infty$.
\section{Problem formulation}

Consider the Markov Decision Process (MDP) $\left(X,A, \{A(x)\mid x\in X\}, \trsknl, \alpha, \bar{\ell}\right)$, 
where $\bar{\ell}:\scspace\rightarrow\mR$ is bounded and $\alpha\in(0,1)$.
The expert takes an optimal policy that minimizes the expectation of 
the infinite-time horizen discounted cumulative cost

\begin{subequations}
\begin{align}
  \optm{(\pi_t)}\in \argmin_{(\pi_t)\in \Pi} & \expectation
                  {\sum_{t=0}^\infty\alpha^t\bar{\ell}(\rv{x}_t, \rv{a}_t)},\\
  \text{s.t. }  &\rv{x}_{0}\sim \initdistr,\\
                &\rv{x}_{t+1}\sim \trsknl(\cdot\mid \rv{x}_t, \rv{a}_t),\\
                &\rv{a}_{t} \sim \pi_t(\cdot\mid \rv{x}_{t}),
\end{align}\label{eq: original OC}
\end{subequations}
where $\Pi:=\left\{ (\pi_t)_{t=0}^\infty \mid 
\pi_t:X\rightarrow \pdMeasure{A},\, \pi_t(A(x)\mid x)= 1, \right.$ $\left.\forall x\in X\right\}$
\footnote{
  Since there always exists an optimal Markov policy \citep[Prop.~9.1]{bertsekas1996stochastic}, 
  we only consider Markov policies here.
}
is the set of feasible policies.
Moreover, we assume $A(x)=A$ for all $x\in X$, which means the action set is independent of the state.
Furthermore, we impose the following mild assumptions on the system:
\begin{assumption}[Transition Kernel]\label{Asmp: Transform kernel cont.}
  The transition kernel $\trsknl$ is weak Feller \citep[App.~C, Def.~C.3]{hernandez2012discrete}. 
  More precisely, viewing $\trsknl$ as a push-forward mapping 
    $\trsknl:\sgnMeasure{\scspace}\rightarrow \sgnMeasure{X}$, 
    this is equivalent to saying that $\adj{\trsknl}(\contfnc{X})\subset \contfnc{\scspace},$
    where $\adj{\trsknl}:\boundedfunctional{X}\rightarrow\boundedfunctional{\scspace}$ is its adjoint operator.
\end{assumption}

\begin{assumption}[State-action set]\label{Asmp: Compact space}
  The state space $X$ is compact and Archimedean \citep[Def.~3.18]{laurent2008sums}.
  Moreover, the feasible action set $A(x) = A\subset \mR ^{n_a}$ is 
    non-empty and independent of the state, 
    and $A$ is also a compact Archimedean set.
\end{assumption}

\begin{assumption}[Cost function]\label{Asmp: Cont. cost features}
  The cost function has a given structure: $\ell=\theta_{\ell}^{\top}\varphi$, 
  where $\varphi:\scspace\rightarrow \mR ^{n_\ell}$ is continuous and known a priori. 
  Consequently, the ``true" cost of the expert $\bar{\ell} = \theta_\ell^\top\varphi$ 
  is bounded with respect to supremum norm.
\end{assumption}

Notably, \eqref{eq: original OC} is nonlinear and such nonlinearity would cause issues such as local minimal if applying standard IOC techniques \citep{rickenbach2024inverse,keshavarz2011imputing,molloy2018finite,molloy2020online,jin2019inverse}.
Hence we are motivated to acquire a linear reformulation of the problem \eqref{eq: original OC}, we adopt the occupation measure framework \citep{hernandez2012discrete}.
In particular, the occupation measure $\mu_t\in\pdMeasure{\scspace}$ is defined as
\begin{align*}
  & \mu_t(\Gamma_x\times \Gamma_a):= \expectation{\mathbb{I}_{\Gamma_x\times \Gamma_a}(\rv{x}_t, \rv{a}_t)},
    &\forall \Gamma_x\times\Gamma_a\in \mathcal{B}(\scspace).
\end{align*}
Consequently, the forward problem \eqref{eq: original OC} can be equivalently expressed as
\begin{subequations}
\begin{align}
  \optm{V}(\initdistr)&:=\min_{ (\mu_t)_{t=0}^\infty } \sum_{t=0}^{\infty} 
                      \langle \bar{\ell}, \alpha\mu_t \rangle,\\
  \text{s.t. }  & \proj_x\mu_0 = \initdistr,\\
                & \proj_x\mu_{t+1} = \trsknl\mu_t,\label{eq: occupation measure evolution}\\
                & \mu_t \in \pdMeasure{\scspace},
\end{align}\label{eq: OC in Deterministic Model}
\end{subequations}
where $\proj_x:\sgnMeasure{\scspace}\rightarrow\sgnMeasure{X}$ maps $\mu_t$ to its marginal distribution on $X$, namely,
\begin{align*}
  (\proj_x\mu_t)(\Gamma_x) := \mu_t(\Gamma_x, A),\quad\forall \Gamma_x\in \mathcal{B}(X).
\end{align*}
Since we only consider Markov policies $(\pi_t)_{t=0}^\infty\in\Pi$, 
by \citep[Prop.~9.2]{bertsekas1996stochastic}, 
for any feasible policy $(\pi_t)_{t=0}^\infty$ that is feasible to \eqref{eq: original OC}, 
there exists a corresponding occupation measure 
$(\mu_t)_{t=0}^\infty$ that is feasible to \eqref{eq: OC in Deterministic Model},
and the converse also holds.
Moreover, by \citep[Cor.~9.5.1]{bertsekas1996stochastic}, 
\eqref{eq: OC in Deterministic Model} and \eqref{eq: original OC} 
have the same optimal value and for corresponding policies 
$(\pi_t)_{t=0}^\infty$ and $(\mu_t)_{t=0}^\infty$, 
$(\mu_t)_{t=0}^\infty$ is optimal to \eqref{eq: OC in Deterministic Model} 
if and only if $(\pi_t)$ is optimal to \eqref{eq: original OC}.
This means that it is equivalent for us to consider 
the infinite-dimensional linear formulation \eqref{eq: OC in Deterministic Model} 
of the original nonlinear one \eqref{eq: original OC}.

Next, consider the value iteration operator 
    $T:\boundedfunctional{X}\rightarrow\boundedfunctional{X}$ defined by
\begin{align*}
  T(V)(x) := \inf_{a\in A} (\bar\ell + \alpha \adj{\trsknl}V)(x,a).
\end{align*}
By Assumptions~\ref{Asmp: Transform kernel cont.}, \ref{Asmp: Compact space}, and \ref{Asmp: Cont. cost features}, 
$ (\bar\ell + \alpha\adj\trsknl V)\in \contfnc{\scspace} $ is uniformly continuous if $V\in \contfnc{X}$. 
Thus, we have $ T(\contfnc{X})\subset\contfnc{X}$. 
Now let $ V_0=0$, it follows that $V_{t+1}:=T(V_t)\in\contfnc{X}$ for all $t\ge 0$.
Since $T$ is a contraction mapping on $\boundedfunctional{X}$ 
    with respect to the supremum norm \citep[p.~52]{hernandez2012discrete},
  the limit $\bar{V}:=\lim_{t\rightarrow\infty}V_t \in \contfnc{X}$ exists and is the unique fixed point.
By \citep[Prop.~9.17]{bertsekas1996stochastic}, 
  there exists an optimal deterministic stationary Markov policy for \eqref{eq: original OC}, 
  and
\begin{align}
  \int_X \bar{V}(x)\initdistr(\drm x)=\optm{V}(\initdistr).\label{eq: Vstar-with-Vbar}
\end{align}
Thus, without loss of generality, we introduce the following assumption.
\begin{assumption}[Expert policy]\label{Asmp: D.S.M. policy} 
  The expert adopts a stationary Markov policy, i.e.,
  $ \pi_t(\rv{a}\mid \rv{x}_{0:t}, \rv{a}_{0:t-1})=\bar{\pi}(\rv{a}\mid\rv{x}_t)$.
\end{assumption}
In addition, we assume there exist state-action pairs for \eqref{eq: OC in Deterministic Model} 
that are not optimal, and that these pairs have a non-zero (Lebesgue) measure in $X\times A$. In particular, we make the following assumption.

\begin{assumption}[Regularity condition]\label{Asmp: non-trivial cost}
  There exists $(x,a)\in\scspace$ such that $ (\bar{\ell} + \alpha\adj\trsknl \bar{V})(x,a) > \bar{V}(x)$.
  Equivalently, since $ (\bar{\ell} + \alpha\adj\trsknl \bar{V} - \bar{V})\in \contfnc{\scspace}$, we have
  \begin{align*}
    \int_{\scspace} (\bar{\ell} + \alpha\adj\trsknl \bar{V} - \bar{V}) \drm x \drm a > 0.
  \end{align*}
\end{assumption}
The above assumption serves as a regularity condition for the IOC problem to come.
Such assumption avoid trivial and uninformative situations like ``when $\theta_\ell=0$, 
any policy and trajectory are optimal''.

Further, let $\mu_t^{\bar{\pi}}$ be the occupation measure induced by the ``true" optimal policy
 $\bar{\pi}$ of \eqref{eq: original OC}, and consider the following problem
\begin{subequations}
\begin{align}
  \min_{\mu_\Sigma\in \sgnMeasure{\scspace}} &\langle \bar{\ell},\mu_\Sigma \rangle,\\
  \text{s.t. } & (\proj_x-\alpha Q)\mu_\Sigma = \initdistr,\\
               & \mu_\Sigma \in \pdMeasure{\scspace}.
\end{align}\label{eq: linear form OC primal}
\end{subequations}
Note that $\mu_\Sigma^{\bar{\pi}}:=\sum_{t=0}^\infty \alpha^t \mu_t^{\bar{\pi}}$ 
is still feasible for such an infinite-dimensional linear program.
This implies the fact that $(\mu_t^{\bar\pi})_{t=0}^\infty$ 
is optimal to \eqref{eq: OC in Deterministic Model} if and only if 
$\mu_\Sigma^{\bar\pi}$ is optimal to \eqref{eq: linear form OC primal}.
And the dual of \eqref{eq: linear form OC primal} takes the form
\begin{subequations}
\begin{align}
  \max_{V \in \boundedfunctional{X}} &\langle V,\initdistr \rangle,\\
  \text{s.t. } & (\bar{\ell} - \adj{(\proj_x-\alpha Q)}V)(x,a) \ge 0, 
                                \forall (x,a)\in\scspace.
\end{align}\label{eq: linear form OC dual}
\end{subequations}
Since it holds that $\bar{\ell} - \adj{(\proj_x-\alpha Q)}\bar{V} \ge T(\bar{V})-\bar{V}=0$, 
$\bar{V}$ is also feasible for \eqref{eq: linear form OC dual}, and we have
\begin{align*}
  \langle \bar{V},\initdistr \rangle=\int_X \bar{V}(x)\initdistr(\drm x)
      \stackrel{\text{\eqref{eq: Vstar-with-Vbar}}}{=} \optm{V}(\initdistr) = \langle \bar{\ell},\mu_\Sigma^{\bar{\pi}} \rangle.
\end{align*}
Thus, strong duality holds between \eqref{eq: linear form OC primal} and \eqref{eq: linear form OC dual}.

Since \eqref{eq: linear form OC primal} is convex, 
this further implies that its Karush-Kuhn-Tucker (KKT) conditions are both sufficient and necessary.
In view of the fact that \eqref{eq: original OC} and \eqref{eq: OC in Deterministic Model} 
are equivalent, given the ``true'' optimal distribution $(\mu_t^{\bar{\pi}})$ 
of the state-action pairs, we can consider the structure of KKT conditions 
of \eqref{eq: linear form OC primal} instead of the nonlinear ones of \eqref{eq: original OC}.

Next, we formulate the problem of IOC.
Notably, we only have access to a finite time-length dataset of expert demonstrations in practice. 
To proceed, we need to make an assumption on the ``information'' contained in the finite time-length dataset.

\begin{assumption}[Persistent excitation]\label{Asmp: enough state observed}
  The observed trajectory data length $N$ is a positive integer such that 
  $\mu_\Sigma^{\bar\pi,x}$ is absolutely continuous with respect to $\mu^{N,\bar{\pi},x}$, 
  where $\mu_\Sigma^{\bar\pi,x}:=\mathcal{P}_x\mu_\Sigma^{\bar\pi}$, $\mu^{N,\bar\pi,x}:=\mathcal{P}_x\mu^{N,\bar\pi}$, 
  $\mu^{N,\bar\pi}:=\gamma\sum_{t=0}^{N-1}\alpha^t\mu_t^{\bar\pi}$ and $\gamma := \frac{1-\alpha}{1-\alpha^N}$ is 
  a normalization constant that ensures $\mu^{N,\bar\pi}$ is a probability measure.
\end{assumption}
Intuitively speaking, the above assumption guarantees that the observed trajectory data of length 
$N$ does not miss any state that has non-negligible probability 
over the infinite discounted horizon under the expert policy $\bar{\pi}$.

Moreover, in practice, the observations are usually corrupted by measurement noise that arises 
from sensing errors or imperfect data logging.
Thus, we model each state-action observation to be a consequence of the contamination 
of additive noise, which is independent of the state-action pairs 
and it is identically distributed over the time instants according to a known distribution 
(which can be obtained by sensors calibration).
\begin{assumption}[Observation noise]\label{Asmp: i.i.d. obsv}
The observation of a finite-time trajectory $(\rv{y}_t)_{t=0:N}$ takes the form of
\begin{align*}
    \rv{y}_{t} := \begin{bmatrix}
      \rv{y}_{t,x}\in\mR^{n_x}\\ \rv{y}_{t,a}\in\mR^{n_a}
    \end{bmatrix} = \begin{bmatrix} \rv{x}_{t} \\ \rv{a}_{t} \end{bmatrix} + \underbrace{\begin{bmatrix}
      \rv{v}_{t,x}\\ \rv{v}_{t,a}
    \end{bmatrix}}_{\rv{v}_{t}},
\end{align*}
where $(\rv{x}_{t},\rv{a}_{t})$ denotes the ``true" state-action pair generated by the expert policy; 
$\rv{v}_{t} \sim \nu$ is observation noise and is independent of $(\rv{x}_t, \rv{a}_t)$.
Moreover, $\nu$ is a known non-degenerate probability distribution characterized by the intrinsics of the measurement device, 
with a strictly positive definite covariance $\Sigma_v$.
The dataset comprises $M$ i.i.d. trajectories of length $N+1$, i.e., $\left\{(y_{t,i})_{t=0}^{N}\right\}_{i=1}^{M}$, 
where $y_{t,i}$ is the realization of the i.i.d. random variable $\rv{y}_{t,i}$ for $i=1,\ldots,M$, and for any given $t$, $\rv{y}_{t,i}$ has the same distribution as $\rv{y}_t$ for all $i$.
\end{assumption}

Equipped with the above set-ups, we are ready to present the IOC problem formulation.
\begin{problem}\label{prb: main problem}
  Let a discrete-time MDP \eqref{eq: original OC} be governed by an expert policy $\bar\pi$ that 
  satisfies Assumption~\ref{Asmp: Compact space}- \ref{Asmp: D.S.M. policy} and \ref{Asmp: non-trivial cost}. 
  And let the collected observation data set $\left\{(y_{t,i})_{t=0}^{N}\right\}_{i=1}^{M}$ 
  satisfies Assumption~\ref{Asmp: enough state observed} and \ref{Asmp: i.i.d. obsv}.
  Recover the underlying parameters $\theta_\ell$ of the cost function, 
  to which the expert policy $\bar{\pi}$ is optimal.
\end{problem}

\section{The algorithm construction}
In this section, we propose the algorithm to 
recovery the underlying cost function by the given noisy observation. 

\subsection{The noiseless case}
Our prior work \citep{wang2025consistent} has established the IOC formulation for noise-free observations; 
and here we briefly summarize the key elements for completeness.
As mentioned earlier, since strong duality holds between \eqref{eq: linear form OC primal} and \eqref{eq: linear form OC dual},
given the policy $\bar{\pi}$, it is globally optimal if and only if 
the corresponding $\mu^{\bar{\pi}}_{\Sigma}:=\sum_{t=0}^\infty\mu_t^{\bar\pi}$ satisfies the following Karush-Kuhn-Tucker condition.
\begin{subequations}
  \begin{equation}\label{eq: measure_forward_transition_constraints}
      (\proj_x-\alpha Q)\mu^{\bar{\pi}}_\Sigma = \initdistr,\mu^{\bar{\pi}}_\Sigma \in \pdMeasure{\scspace},
  \end{equation}
  \begin{equation}\label{eq: dual constraints}
    (\bar{\ell} - \adj{(\proj_x-\alpha Q)}V)(x,a) \ge 0, 
                                \forall (x,a)\in\scspace,    
  \end{equation}
  \begin{equation}\label{eq: complementary relaxation}
    \langle\bar{\ell} - \adj{(\proj_x-\alpha Q)}V, \mu^{\bar{\pi}}_\Sigma\rangle = 0.
  \end{equation}
\end{subequations}
Notably, given the expert's noise-free occupation measure $\mu_{\Sigma}^{\bar{\pi}}$, 
\eqref{eq: measure_forward_transition_constraints} is always satisfied.
Then, \eqref{eq: dual constraints} and \eqref{eq: complementary relaxation} are satisfied with $\ell=\bar{\ell}$. 
Conversely, if a $\hat\ell$ satisfies \eqref{eq: dual constraints} and \eqref{eq: complementary relaxation},
then $\bar\pi$ is also its optimal policy. 
Hence, the IOC problem reduces to finding $\ell$ that satisfies both \eqref{eq: complementary relaxation} and \eqref{eq: dual constraints}.
To ease the notation, we henceforth omit the superscript $\bar\pi$ and write $\mu_{\Sigma}$ as the expert's occupation measure.
And similarly, we write $\mu^{N}$ instead of $\mu^{N,\bar\pi}$.

In particular, under Assumption~\ref{Asmp: enough state observed}, the complementary slackness \eqref{eq: complementary relaxation} 
admits an equivalent formulation using the finite-horizon measure 

$\mu^N$, namely,
\begin{align}\label{eq: finite time complementary relaxation}
  \langle \ell - \adj{(\proj_x-\alpha Q)}V, \mu^N \rangle = 0.
\end{align}

We formally state this in the following proposition.

\begin{proposition}\label{prop: finite time equivalence}
  Recall the notation $\mu_\Sigma^x := \proj_x \mu_\Sigma $ and 
$\mu^{N,x} := \proj_x \mu^N$ and suppose Assumption~\ref{Asmp: enough state observed} holds.
Then for any feasible pair $(\ell, V)$ that satisfies \eqref{eq: dual constraints}, 
\eqref{eq: complementary relaxation} is equivalent to \eqref{eq: finite time complementary relaxation}.
\end{proposition}
\begin{proof}
  Since $\mu^{N,x}$ is a finite sum of the same marginals $\proj_x\mu_t$ 
  that constitute $\mu_\Sigma^x$, each with positive weight,
  it follows that $\mu^{N,x}\ll \mu_\Sigma^x$.
  By Assumption \ref{Asmp: enough state observed}, this further implies that $\mu^{N,x}$ and $\mu_\Sigma^x$ are mutually absolutely continuous. 
  In addition, let $\rho := \frac{\drm \mu_\Sigma^x}{\drm \mu^{N,x}}$ denote the Radon-Nikodym derivative, 
  and it holds that $0 < \rho(x) < \infty$ for $\mu^{N,x}$-almost every $x$.
  Because both occupation measures arise from the same (possibly stochastic) Markov policy, 
  their disintegrations share the same conditional kernel $\pi(\drm a \mid x)$:
  \[
  \mu_\Sigma(\drm x, \drm a) = \mu_\Sigma^x(\drm x)\,\pi(\drm a\mid x),\quad
  \mu^{N}(\drm x, \drm a) = \mu^{N,x}(\drm x)\,\pi(\drm a\mid x).
  \]
  For any feasible pair $(\ell, V)$ that satisfies \eqref{eq: dual constraints}, let
  $h(x,a):=(\ell - \adj{(\proj_x-\alpha Q)}V)(x,a) \ge 0$.
  Then it holds that
  \begin{align*}
    \langle h, \mu_\Sigma \rangle 
      &= \int_X \int_A h(x,a)\,\pi(\drm a\mid x)\,\mu_\Sigma^x(\drm x) \\
      &= \int_X \rho(x) \Big[\int_A h(x,a)\,\pi(\drm a\mid x)\Big] \mu^{N,x}(\drm x) = \langle h\cdot\rho,\; \mu^{N} \rangle,
  \end{align*}
  and hence it also holds that $\langle h\cdot \frac{1}{\rho},\mu_\Sigma\rangle = \langle h,\mu^N\rangle$.
  Since it also holds that $0<\rho<\infty$ $\mu^{N}$-a.s., it follows that $\langle h, \mu_\Sigma \rangle = 0$ 
  if and only if $\langle h, \mu^{N} \rangle = 0$. Therefore,
  \begin{align}\label{eq: support argument}
    \langle \ell \!-\! \adj{(\proj_x\!-\!\alpha Q)}V, \mu_\Sigma \rangle = 0
    \!\iff\! \langle \ell \!-\! \adj{(\proj_x\!-\!\alpha Q)}V, \mu^N \rangle = 0.
  \end{align}
\end{proof}
\begin{remark}[Extension of our prior work]
  Compared to \citep[Prop.~3.3]{wang2025consistent}, which requires a deterministic expert policy, 
  Proposition~\ref{prop: finite time equivalence} only relies on the existence of Radon-Nikodym derivative, and it
  applies to any Markov policy that satisfies Assumption~\ref{Asmp: enough state observed}.
\end{remark}
Proposition~\ref{prop: finite time equivalence} guarantees that we can replace $\mu_\Sigma$ with $\mu^N$ in \eqref{eq: complementary relaxation}.
Such gurantee is useful to constructing the IOC algorithm since its goal is to recover the underlying parameter $\theta_\ell$ 
in the cost function $\ell$ from the observed state-action trajectories $\{(y_{t,i})_{t=0}^N\}_{i=1}^M$ with finite length.
To this end, we propose to solve the following optimization problem.
\begin{subequations}\label{eq: ioc infinite dim}
  \begin{align}
      \min_{\psi, \theta_\ell\in \mathbb{R}^{n_\ell}, V} &\;\;\langle \psi, \mu^N\rangle,\\
      &\;\; \psi = \theta_\ell^\top \varphi + \alpha \adj{\trsknl} V - V,\label{eq: psi-linear-combi}\\
      &\;\; \int_{\scspace} \psi(x,a)\drm x\drm a \ge 1,\label{eq: psi-non-trivial}\\
      &\;\; \psi\in\contfnc{\scspace},\psi(x,a)\ge 0, \forall (x,a)\in\scspace,\label{eq: psi-in-C+}\\
      &\;\; V\in \contfnc{X} \label{eq: V-in-C},\\
      &\;\; \|V\|_\infty \le \beta_V, \|\theta_\ell\|_1\le \beta_\ell. \label{eq: norm-finite-constr}
  \end{align}
\end{subequations}
In particular, to avoid the trivial and uninformative solution mentioned in Assumption \ref{Asmp: non-trivial cost}, 
we add the regularity condition in Assumption \ref{Asmp: non-trivial cost} as the constraint \eqref{eq: psi-non-trivial}. 
And without losing generality, we set $\int_{\scspace} \psi(x,a)\drm x\drm a$ to be greater than 1 
since both $\theta_\ell$ and $V$ can anyway be scaled by a positive coefficient and the expert policy $\bar\pi$ still remains optimal.
By a similar argument as in \citep[Prop.~3.2, Prop.~3.3]{wang2025consistent}, the following property holds.
\begin{proposition}\label{prop: optimal value zero}
  Under Assumptions~\ref{Asmp: Cont. cost features}--\ref{Asmp: i.i.d. obsv}, 
  \eqref{eq: ioc infinite dim} attains its minimum with optimal value zero.
  Moreover, a cost $\ell=\theta_\ell^\top \varphi$ makes $\bar\pi$ optimal 
  if and only if there exists $V$ and $\psi$ such that 
  $(\psi, \theta_\ell, V)$ solves \eqref{eq: ioc infinite dim}.
\end{proposition}
This proposition shows that solving Problem \ref{prb: main problem} is equivalent to 
solving an infinite-dimensional linear program \eqref{eq: ioc infinite dim}.
Though being convex, \eqref{eq: ioc infinite dim} is still an infinite-dimensional optimization problem. 
To proceed, we need to approximate it with a finite dimensional optimization problem.
By following the approximation procedure in \citep[Sec.~3.2]{wang2025consistent}, 
we construct a finite-dimensional approximation. We briefly explain the procedure here for self-containment.

First, let $P_{d_V}(X)$ denote 
the space of polynomials on $X$ with total degree at most $d_V$, 
and let $r:X\rightarrow \mathbb{R}^{D_V}$ denote a basis of $P_{d_V}(X)$.
Similarly, let $\phi:\scspace\rightarrow \mathbb{R}^{D_\psi}$ denote
a basis of $P_{d_\psi}(\scspace)$ with $d_\psi \ge d_V$.
To turn the infinite dimensional optimization \eqref{eq: ioc infinite dim} into a finite dimensional one, 
we restrict the optimization variables in \eqref{eq: ioc infinite dim} 
to finite-degree polynomials, namely,
\begin{align*}
  V \approx \theta_V^\top r,\qquad \psi \approx \theta_\psi^\top\phi,
\end{align*}
where $\theta_V$ and $\theta_\psi$ are the coefficients of the polynomials.
Since $(\adj{\proj}r)(x,a):=r(x)$ are polynomials on $\scspace$, there exists $G_1\in\mR^{D_\psi\times D_V}$ such that 
\begin{align}\label{eq: proj r G1 phi}
  \adj{\proj} r = G_1^\top\phi.
\end{align}
We further approximate the continuous feature functions $\varphi$ in the cost function 
and the basis functions $r(x)$ forward evolution with the linear combinations of polynomial basis $\phi$. 
Namely,
\begin{align}\label{eq: matrix definitions}
  \varphi \approx H^\top\phi, \quad 
  \adj{\trsknl}r \approx G_2^\top\phi,
\end{align}
where $H\in\mathbb{R}^{D_\psi\times n_\ell}$, and $G_2\in\mathbb{R}^{D_\psi\times D_V}$.\footnote{
  When taking Lagrange interpolation basis with the interpolation nodes $\{\eta_{i}\}_{i=0}^{D_\psi-1}$ as $\phi$,
  matrices $G_1$, $G_2$ and $H$ are computed by evaluating the function values of $r, \adj{Q}r$ and $\varphi$ at the interpolation nodes.
  For example, the $(i,j)$-th entry is $(G_1)_{[i,j]} = r_i(\proj_x \eta_{j})$, 
  and $(G_2)_{[i,j]} = \int r_i(x')Q(\drm x'\mid\eta_{j})$.
}
Now that $\psi$, $\varphi$, $\adj{\trsknl}V$ and $V$ are approximated under the same basis, 
the constraints \eqref{eq: psi-linear-combi} and \eqref{eq: psi-non-trivial} can be approximated as
\begin{align*}
  \theta_\psi = H\theta_\ell + \alpha G_2 \theta_V - G_1 \theta_V,\quad
  \theta_\psi^\top \underbrace{\int_{\scspace} \phi(x,a) \drm x \drm a}_{d} \ge 1.
\end{align*}
Consequently, the infinite-dimensional optimization \eqref{eq: ioc infinite dim} is transformed to 
the following finite-dimensional approximation.
\begin{subequations}\label{eq: finite dim ioc in expectation}
\begin{align}
  \min_{\theta_\psi, \theta_\ell, \theta_V}&\; \bar{m}^\top\theta_\psi,\label{eq: ioc_algorithm_obj_fun}\\
  & \theta_\psi = \underbrace{\begin{bmatrix}
    H, \alpha G_2 - G_1
  \end{bmatrix}}_{\Xi_\psi \in \mR^{D_\psi\times (n_\ell + D_V)}}\begin{bmatrix}
    \theta_\ell\\\theta_V
  \end{bmatrix},\label{eq: finite dim ioc-constr-1}\\
  &d^\top \theta_\psi \ge 1,\\
  &\theta_\psi^\top\phi\ge 0,\label{eq: finite dim ioc-constr-3}\\
  &\|\theta_V\|_1\le \beta_V',\|\theta_\ell\|_1\le \beta_\ell,\label{eq: finite dim ioc-constr-4}
\end{align}
\end{subequations}
where 
\begin{align}\label{eq: moment vector definition}
  \bar{m} := \int_{\scspace} \phi(\eta) \, \mu^N(\drm\eta) = \mathbb{E}_{\rv{\eta} \sim \mu^N}[\phi(\rv{\eta})]
\end{align}
is the moment vector under occupation measure $\mu^N$. 
And since the norm bounds $\beta_V$ and $\beta_\ell$ in \eqref{eq: norm-finite-constr} can be chosen arbitrarily large 
in practice, we can assume that they are inactive for the optimal solution, and thus replacing the $\beta_V$ by a sufficient large $\beta_V'$ in \eqref{eq: finite dim ioc-constr-4}.
Though being finite-dimensional, the constraint \eqref{eq: finite dim ioc-constr-3} need extra care so that 
we can solve \eqref{eq: finite dim ioc in expectation} numerically. We elaborate on this in Sec.~\ref{sec: experiments}.
On the other hand, note that we also need the moment vector $\bar{m}$ to solve \eqref{eq: finite dim ioc in expectation}, 
which need to be estimated with the noisy observations.
This is addressed in the next subsection.

\subsection{Moment estimation from noisy observations}\label{sec: moment estimator}
To estimate the moment vector $\bar{m}$ from the corrupted trajectory observation $\{(y_{t,i})_{t=0}^N\}_{i=1}^M$, 
we develop a statistical procedure that leverages both the additive noise structure and the system dynamics.

In this section, for notational simplicity, 
the analysis is based on monomial bases. Namely,
both $\phi:\scspace\rightarrow \mathbb{R}^{D_\psi}$ and $r:X\rightarrow \mathbb{R}^{D_V}$ 
are assumed to be monomial bases whose degrees do not exceed $d_\psi$ and $d_V$, 
and they have dimensions $D_\psi$ and $D_V$ respectively.
This is not restrictive, since we can always obtain the moments in other bases (e.g. Lagrange interpolation basis) via a linear transformation of the moments in monimial basis.\footnote{
  The analysis in this section yields moments estimation in monomial basis.
  If the basis $\phi$ in \eqref{eq: moment vector definition} is not monomial, we can obtain its corresponding moments estimation via the following steps.
  Let $\phi_{\mathrm{mono}}$ denote the monomial basis 
  and $\bar{m}_{\mathrm{mono}} := \int \phi_{\mathrm{mono}}(\eta) \mu^N(\drm\eta)$ 
  be the corresponding moment vector.
  Since $\phi$ is a polynomial basis, 
  there exists an invertible change-of-basis matrix $B \in \mathbb{R}^{D_\psi \times D_\psi}$ 
  such that $\phi = B \phi_{\mathrm{mono}}$.
  The moment vector in \eqref{eq: moment vector definition} is then obtained 
  via $\bar{m} = B^\top\bar{m}_{\mathrm{mono}}$.
  Hence the subsequent GMM procedure estimates $\bar{m}_{\mathrm{mono}}$ from noisy data,
  and $\bar{m}$ is recovered by applying the known linear transformation $B$.
}
More precisely, for any $\eta = (x, a) \in \mR^{n_x} \times \mR^{n_a}$, let 
\begin{align*}
  \phi(\eta) = \left(\eta^{\mathbf{d}}\right)_{\mathbf{d} \in \mathcal{D}_\psi},
\end{align*}
where $\mathcal{D}_\psi := \{\mathbf{d} \in \mathbb{N}^{n_x+n_a} \mid \|\mathbf{d}\|_1 \leq d_\psi\}$ 
is the set of multi-indices with total degree at most $d_\psi$,
and $\eta^{\mathbf{d}} := \eta_1^{d_1} \eta_2^{d_2} \cdots \eta_{n_x+n_a}^{d_{n_x+n_a}}$ 
for $\mathbf{d} = [d_1, \ldots, d_{n_x+n_a}]^\top\in\mathcal{D}_\psi$.
Our goal is to estimate the moment vector $\bar{m}$
from the noisy observations $\{\{\rv{y}_{t,i}\}_{t=0}^N\}_{i=1}^M$.

\subsubsection{Moment conditions from additive noise}
For each multi-index $\mathbf{d} \in \mathcal{D}_\psi$, 
the raw moment of the observed state-action pair at time-step $t$ takes the form
\begin{align*}
  \expectation{\rv{y}_{t}^\mathbf{d}} 
  &= \expectation{(\rv{\eta}_{t}+\rv{v}_{t})^\mathbf{d}},
\end{align*}
where $\rv{\eta}_{t} = [\rv{x}_{t}^\top, \rv{a}_{t}^\top]^\top$ denotes the true (unobserved) state-action pair 
and $\rv{v}_{t}$ is the observation noise.
By the binomial theorem in multi-index notation, it holds that
\begin{align}\label{eq: binomial expansion}
  (\rv{\eta}_t+\rv{v}_t)^\mathbf{d} 
  = \sum_{\mathbf{d}'} C(\mathbf{d}, \mathbf{d}') \rv{\eta}_t^{\mathbf{d}'} \rv{v}_t^{\mathbf{d}-\mathbf{d}'},
\end{align}
and the multi-index binomial coefficient takes the form
\begin{align*}
  C(\mathbf{d}, \mathbf{d}') := \prod_{j=1}^{n_x+n_a} C(d_j, d_j'),
\end{align*}
where $C(n, k) = \binom{n}{k} = \frac{n!}{k!(n-k)!}$ denotes the standard binomial coefficient, i.e., ``$n$ choose $k$''.
Moreover, by Assumption~\ref{Asmp: i.i.d. obsv}, $\rv{v}_{t}$ is independent of $\rv{\eta}_{t}$.
Taking expectations on both sides of \eqref{eq: binomial expansion} yields
\begin{align}\label{eq: moment condition}
  \expectation{\rv{y}^\mathbf{d}_{t}} 
  = \sum_{\mathbf{d}'} C(\mathbf{d}, \mathbf{d}') 
    \expectation{\rv{v}^{\mathbf{d}-\mathbf{d}'}_{t}} \expectation{\rv{\eta}^{\mathbf{d}'}_{t}},
\end{align}
where the noise moments $\expectation{\rv{v}^{\mathbf{d}^{\prime\prime}}_{t}}$ are known a priori
  by Assumption~\ref{Asmp: i.i.d. obsv}.
Aggregating the powered observations with the discount structure, we define
\begin{align}\label{eq: sample moment}
  \rv{m}_{\mathbf{d}}^{\mathrm{obs}} := \gamma\sum_{t=0}^{N-1} \alpha^t \rv{y}_{t}^\mathbf{d}.
\end{align}
Taking the expectation on both hand sides of \eqref{eq: sample moment}, plugging \eqref{eq: moment condition} in and in view of the fact that $\rv{v}_t\sim\nu$ is stationary and independent of $\rv{\eta}_t$ from Assumption \ref{Asmp: i.i.d. obsv}, we have
\begin{align*}
  \expectation{\rv{m}_\mathbf{d}^{\mathrm{obs}}} &= \gamma\sum_{t=0}^{N-1}\alpha^t \sum_{\mathbf{d}'} C(\mathbf{d}, \mathbf{d}') 
    \expectation{\rv{v}^{\mathbf{d}-\mathbf{d}'}_{t}} \expectation{\rv{\eta}^{\mathbf{d}'}_{t}}\\
    &=\sum_{\mathbf{d}'}C(\mathbf{d},\mathbf{d}')\expectation{v^{\mathbf{d}-\mathbf{d}'}}\underbrace{\left(\gamma\sum_{t=0}^{N-1}\alpha^t\expectation{\rv{\eta}_t^{\mathbf{d}'}}\right)}_{\bar{m}_{\mathbf{d}'}},
\end{align*}
where $\rv{v}$ has the same distribution $\nu$ as $\rv{v}_t$ for all $t$.
The above equation relates the observed statistics to the component $\bar{m}_{\mathbf{d}'}$ in the unknown moment vector $\bar{m}$ defined in \eqref{eq: moment vector definition}.
Furthermore, stacking all multi-indices $\mathbf{d}_i \in \mathcal{D}_\psi$ in the ascending order defined by $\partialorder$, 
i.e., the graded lexicographical order.
And recall \eqref{eq: moment vector definition},
we can rewrite the above equation in compact form as
\begin{align}\label{eq: moment system}
  \mathbb{E}[\underbrace{\begin{bmatrix}\rv{m}^{\mathrm{obs}}_{\mathbf{d}_1}\\\vdots\\\rv{m}^{\mathrm{obs}}_{\mathbf{d}_{n_x+n_a}}\end{bmatrix}}_{\rv{m}^{\mathrm{obs}}}]
  = \Phi_\nu \underbrace{\begin{bmatrix}
  \gamma\sum_{t=0}^{N-1}\alpha^t \expectation{\rv{\eta}_t^{\mathbf{d}_1}}\\
  \vdots\\
  \gamma\sum_{t=0}^{N-1}\alpha^t \expectation{\rv{\eta}_t^{\mathbf{d}_{n_x+n_a}}}
  \end{bmatrix}}_{\bar{m}},
\end{align}
where the components of matrix $\Phi_\nu \in \mathbb{R}^{D_\psi \times D_\psi}$ takes the form
\begin{equation*}
  \Phi_{v, (\mathbf{d}, \mathbf{d}')} = \begin{cases}
    C(\mathbf{d}, \mathbf{d}')\expectation{\rv{v}^{\mathbf{d}-\mathbf{d}'}}, &\text{if } \mathbf{d}'\partialorder \mathbf{d};\\
    1, &\text{if } \mathbf{d}' = \mathbf{d};\\
    0, &\text{if } \mathbf{d}\partialorder \mathbf{d}'.
  \end{cases}
\end{equation*}
Notably, $\Phi_\nu$ is lower triangular with unit diagonal entries, 
hence it is invertible.
In view of \eqref{eq: sample moment} and denote $f_v( \rv{m}^{\mathrm{obs}}; m ) := \rv{m}^{\mathrm{obs}} - \Phi_\nu m$, we obtain the moment condition
\begin{align}\label{eq: moment function}
  \expectation{f_v( \rv{m}^{\mathrm{obs}}; \bar{m} )}=0
\end{align}
when $m=\bar{m}$.

\subsubsection{Additional moment constraints from system dynamics}
In addition to the moment information induced by observations,
the transition dynamics also impose additional structure on the occupation measure. 
More specifically, the occupation measure evolution \eqref{eq: occupation measure evolution} also implies certain moment conditions.
Let the time-shifted state occupation measure and its associated moment vector be
\begin{align*}
  \mu^{N,x+}\!\!:=\! \gamma\!\!\sum_{t=1}^{N} \alpha^{t-1}(\proj_x \mu_t),\quad 
  \bar{m}^{x+}\!:=\! \langle r, \mu^{N,x+} \rangle \!=\! \mathbb{E}_{\rv{x} \sim \mu^{N,x+}}[r(\rv{x})].
\end{align*}
Plugging \eqref{eq: occupation measure evolution} in the above equation yields
\begin{align*}
  \mu^{N,x+} = \gamma\sum_{t=1}^{N} \alpha^{t-1}(\proj_x \mu_t)
    = \gamma\sum_{t=0}^{N-1} \alpha^{t}\trsknl \mu_t =\trsknl \mu^N.
\end{align*}
Using the monomial basis $r$ of $P_{d_V}(X)$ as test functions, it holds that
\begin{align*}
  \bar{m}^{x+} = \langle r, \mu^{N,x+}\rangle = \langle r, \trsknl\mu^N\rangle = \langle \adj{\trsknl}r, \mu^N\rangle.
\end{align*}
And since $(\adj{\trsknl}r)(x,a) := \int_X r(x')\trsknl(\drm x'\mid x,a)$ is approximated by $G_2^\top\phi(x,a)$ in \eqref{eq: matrix definitions}, it follows from the above equation that
\begin{align}\label{eq: moment approximation}
  \bar{m}^{x+} = \langle \adj{\trsknl}r, \mu^N\rangle \approx 
  \int_{\scspace} G_2^\top\phi(x,a)\,\mu^N(\drm x,\drm a) = G_2^\top \bar{m}.
\end{align}
Furthermore, recall by Assumption~\ref{Asmp: i.i.d. obsv}, it holds that $\rv{y}_{t,x} = \rv{x}_t + \rv{v}_{t,x}$.
Then for each multi-index 
$\mathbf{d} \in \mathcal{D}_V:=\{\mathbf{d} \in \mathbb{N}^{n_x} \mid \|\mathbf{d}\|_1 \leq d_V\}$, define $\rv{m}_{\mathbf{d}}^{\mathrm{obs}, x+} := \gamma\sum_{t=1}^{N}\alpha^{t-1} \rv{y}_{t,x}^\mathbf{d}$; and
similar to \eqref{eq: moment system}, we stack all multi-indices $\mathbf{d} \in \mathcal{D}_V$ in the ascending order defined by $\partialorder$ and get a compact matrix form
\begin{align}\label{eq: shifted moment system}
  \bar{m}^{\mathrm{obs},x+} :=\expectation{\rv{m}^{\mathrm{obs},x+}} = \Phi_{\nu_x} \bar{m}^{x+},
\end{align}
where $\Phi_{\nu_x} \in \mathbb{R}^{D_V \times D_V}$ is the transformation matrix that corresponds to the addtive state noise distribution, and it is
constructed analogously to $\Phi_\nu$ with the moments of $\rv{v}_{t,x}$.
In view of \eqref{eq: moment approximation} and \eqref{eq: shifted moment system}, it holds that
\begin{align}\label{eq: dynamics moment condition}
  \expectation{\rv{m}^{\mathrm{obs},x+}} \approx \Phi_{\nu_x} G_2^\top \bar{m}.
\end{align}
Therefore, let
\begin{align*}
    f_c(\rv{m}^{\mathrm{obs},x+}; m) := \rv{m}^{\mathrm{obs},x+} - \Phi_{\nu_x} G_2^\top m,
\end{align*}
and together with \eqref{eq: shifted moment system}, when $m=\bar{m}$, it yields a misspecified moment condition due to the polynomial approximation error
\begin{align}\label{eq: f2 definition}
  \expectation{f_c(\rv{m}^{\mathrm{obs},x+}; \bar{m})} = \Phi_{\nu_x}(\bar{m}^{x+} - G_2^\top \bar{m}) =: \delta,
\end{align}
where $\delta$ quantifies the approximation error in \eqref{eq: moment approximation}.

\subsubsection{Combined misspecified GMM estimator}

We now combine the (misspecified) moment conditions \eqref{eq: moment function} and \eqref{eq: f2 definition} 
and construct a two-step GMM estimator for the moment vector $\bar{m}$.
In particular, given the $M$ i.i.d. sampled trajectories $\{(\rv{y}_{t,i})_{t=0}^{N}\}_{i=1}^{M}$, for each trajectory $i$, let
\begin{align*}
  \rv{m}_i^{\mathrm{obs}} = \gamma\sum_{t=0}^{N-1} \alpha^t \phi(\rv{y}_{t,i}), \;
  \rv{m}_i^{\mathrm{obs},x+} = \gamma\sum_{t=1}^{N} \alpha^{t-1} r(\rv{y}_{t,x,i}),
\end{align*}
where $\rv{y}_{t,x,i}$ denotes the state component of observation $\rv{y}_{t,i}$.
Note that the zero order moment is always 1, which contains no information.
In addition, since the first component of $\bar{m}$ corresponds to the zero order moment of $\mu^N$, it also takes the value of 1.  
Thus, we can safely remove them from the moment conditions. More specifically, we let
\begin{align*}
  \rv{b}_i := \begin{bmatrix}
    (\rv{m}_i^{\mathrm{obs}})_{[1:]}\\
    (\rv{m}_i^{\mathrm{obs},x+})_{[1:]}
  \end{bmatrix} - \begin{bmatrix}
    (\Phi_{\nu})_{[1:, 0]}\\
    (\Phi_{\nu_x}G_2^\top)_{[1:, 0]}
  \end{bmatrix}, 
  \Phi := \begin{bmatrix}
    (\Phi_{\nu})_{[1:, 1:]}\\
    (\Phi_{\nu_x}G_2^\top)_{[1:, 1:]}
  \end{bmatrix},
\end{align*}
where $\rv{b}_i$ is i.i.d. as $\rv{b}$, and $(\cdot)_{[1:]}$ deletes the first entry of a vector;
 $(\cdot)_{[1:, 0]}$ removes the first row of a matrix and then takes its first column;
 $(\cdot)_{[1:, 1:]}$ removes the first row and the first column of a matrix.
Moreover, let $m_+ := (m)_{[1:]}$ and denote
\begin{align}\label{eq: stacked moment function}
  f(\rv{b}_i;m_+) := 
    \rv{b}_i - \Phi m_+.
\end{align}
We employ a misspecified GMM procedure \citep{hall2003large} 
to estimate $\bar{m}_+:=(\bar m)_{[1:]}$, namely
\begin{align*}
  \hat{\rv{m}}_+^M := \argmin_{m_+ \in \mathbb{R}^{D_\psi-1}}\;\; \left\|g_M(\{\rv{b}_i\}_{i=1}^M;m_+)\right\|_W^2,
\end{align*}
where
\begin{align}\label{eq: sample moment average}
  g_M(\{\rv{b}_i\}_{i=1}^M;m_+) := \frac{1}{M}\sum_{i=1}^M f(\rv{b}_i;m_+),
\end{align}
and $W$ is a positive definite weight matrix to be specified.
In particular, for a correctly specified model \citep[Def.~1]{hall2003large}, 
taking $W$ as the inverse of the covariance matrix $\Sigma$ of $f(\rv{b},m_+)$ leads to 
an efficient estimator. To this end, we estimate the covariance matrix $\Sigma$ by
\begin{equation}\label{eq: covariance estimate}
  \begin{aligned}
  \hat{\Sigma}\!\! &:=\!\! \frac{1}{M}\sum_{i=1}^M\!\! 
  \left[f(\rv{b}_i;\bar{m}_+) - g_M(\{\rv{b}_i\}_{i=1}^M;\bar{m}_+)\right] 
  \left[f(\rv{b}_i;\bar{m}_+) - g_M(\{\rv{b}_i\}_{i=1}^M;\bar{m}_+)\right]^\top\\
  &= \frac{1}{M}\sum_{i=1}^M \rv{b}_i \rv{b}_i^\top - 
  \left(\frac{1}{M}\sum_{i=1}^M \rv{b}_i\right)\left(\frac{1}{M}\sum_{i=1}^M \rv{b}_i\right)^\top.
  \end{aligned}
\end{equation}
Notice that the ``true'' moment vector $\bar{m}_+$ disappears from the estimation of $\hat{\Sigma}$ and only the observed data $\{\rv{b}_i\}_{i=1}^M$ are needed for the covariance estimation.

However, for our misspecified case, there is no choice of weight matrix that yields an asymptotically efficient estimator 
\citep{hall2003large}.
Therefore, 
we adopt a block-diagonal structure for the weight matrix as
\begin{align}\label{eq: weighted matrix}
  \hat{W} = \left(\begin{bmatrix}
    \left(\hat\Sigma\right)_{[:D_\psi-1, :D_\psi-1]} &0\\
    0 & \left(\hat\Sigma\right)_{[D_\psi-1:, D_\psi-1:]}
  \end{bmatrix}+\lambda I\right)^{-1}
\end{align}
where $(\hat\Sigma)_{[:D_\psi-1, :D_\psi-1]}$ and $(\hat\Sigma)_{[D_\psi-1:, D_\psi-1:]}$
are the sample covariance estimates of 
$(\rv{m}_i^{\mathrm{obs}})_{[1:]}$ and $(\rv{m}_i^{\mathrm{obs}, x+})_{[1:]}$ respectively.
Moreover, this block-diagonal structure also has engineering benefits: 
$(\rv{m}_i^{\mathrm{obs}})_{[1:]}$ and $(\rv{m}_i^{\mathrm{obs}, x+})_{[1:]}$
typically exhibit high cross correlation, which would result 
in poor conditioning of the full covariance matrix estimate $\hat{\Sigma}$.
In addition, a small regularization term $\lambda I$ is added to ensure the numerical stability when taking the matrix inverse.
Consequently, the final estimator of the moment vector $\bar{m}_+$ is given by
\begin{align}\label{eq: moment estimator}
  \hat{\rv{m}}_+^M := \argmin_{m_+ \in \mathbb{R}^{D_\psi-1}} \left\|g_M(\{\rv{b}_i\}_{i=1}^M;m_+)\right\|_{\hat{W}}^2.
\end{align}

\subsection{The complete algorithm}

The complete procedure for IOC with noisy observations 
is summarized in Algorithm~\ref{alg:ioc_noisy}.
In particular, the algorithm consists of two main stages: 
(i) Moment estimation from corrupted data via the two-step GMM estimator \eqref{eq: moment estimator}; (ii) Cost recovery via solving \eqref{eq: finite dim ioc in expectation} with the estimated moments $\hat{m}$.

\begin{algorithm}[!htpb]
    \caption{Inverse Optimal Control with Noisy Observations}
    \label{alg:ioc_noisy}
    \begin{algorithmic}[1]
        \Require Noisy trajectory data $\{(y_{t,i})_{t=0}^{N}\}_{i=1}^{M}$, 
        noise distribution $\nu$ (or its moments $\expectation{\rv{v}^{\mathbf{d}}}$), 
        transition kernel $\trsknl(\cdot\mid x,a)$, 
        polynomial degrees $(d_\psi, d_V)$, 
        discount factor $\alpha$.
        \Ensure Cost parameters $\hat{\theta}_\ell$ and value function coefficients $\hat{\theta}_V$.
        \State Choose bases $\phi$ and $r$ with degrees $d_\psi$ and $d_V$, respectively.
        \State Construct the stacked moment condition $f(b_i;m)$ via \eqref{eq: stacked moment function}.
        \State Estimate the covariance $\hat{\Sigma}$ via \eqref{eq: covariance estimate} 
        and compute $\hat{W}$ via \eqref{eq: weighted matrix}.
        \State Calculate $\hat{m}^M$ by solving GMM \eqref{eq: moment estimator}.
        \State Solve \eqref{eq: finite dim ioc in expectation} 
          with estimated moments $\hat{m}^M$, i.e.,
        \begin{equation}\label{eq: final ioc optimization}
        \begin{aligned}
            &\min_{\theta_\psi, \theta_\ell, \theta_V} \quad \theta_\psi^\top \hat{m}^M, \\
            &\text{s.t. \eqref{eq: finite dim ioc-constr-1}-\eqref{eq: finite dim ioc-constr-4}}
        \end{aligned}
        \end{equation}
        \State \textbf{Return} $(\hat{\theta}_\ell, \hat{\theta}_V)$.
    \end{algorithmic}
\end{algorithm}

\section{Consistency analysis}
In this section, we establish the theoretical consistency of the proposed method.
The first lemma shows that the misspecified GMM estimator 
converges in probability to a limit close to the true moment vector $\bar{m}$, 
with the distance controlled by the polynomial approximation error.
\begin{lemma}\label{lem: estimator-converge-in-p} 
  Under Assumption~\ref{Asmp: i.i.d. obsv}, the estimator $\hat{\rv{m}}_+^M$ converges in probability as $M$ tends to infinity. 
  Moreover, let the limit be denoted as $\tilde{m}_+$, 
  then
  it holds that $\left\| \bar{m}_+ - \tilde{m}_+ \right\|_2 \le C\left\| \adj{\trsknl}r  
- G_2^\top\phi \right\|_\infty$,
  where $C$ is a constant that depends on the observation noise distribution $\nu$ and 
  the ground truth occupation measure $\mu^N$.
\end{lemma}
\begin{proof}
  Note that $\Phi$ has full column rank since $\Phi_\nu$ is invertable.
  Thus, $\Phi^\top \hat{W} \Phi$ is strictly positive definite and hence the weighted least square problem \eqref{eq: moment estimator} has a closed form solution
  \begin{align*}
    \hat{\rv{m}}_+^M = (\Phi^\top \hat{W} \Phi)^{-1}\Phi^\top \hat{W}\left( \sum_{i=1}^M \rv{b}_i/M \right).
  \end{align*}
  Recall $\rv{b}_i$ is i.i.d. as $\rv{b}$, it holds that $ \plim_{M\rightarrow\infty} \sum_{i=1}^M \rv{b}_i/M = \expectation{\rv{b}}$, 
  $\plim_{M\rightarrow\infty} \hat{W} =\bar{W}$, where
  \begin{align}\label{eq: limit of weighted matrix}
    \bar{W}:= \left(\begin{bmatrix}
    \cov\left[ (\rv{b})_{[:D_\psi-1]} \right] &0\\
    0 & \cov\left[ (\rv{b})_{[D_\psi-1:]} \right]
  \end{bmatrix}+\lambda I\right)^{-1}.
  \end{align}
  And since $\hat{\rv{m}}_+^M$ is continuous with respect to $\sum_{i=1}^M \rv{b}_i/M$ and $\hat{W}$, 
  it follows that $\plim_{M\rightarrow\infty} \hat{\rv{m}}_+^M = \tilde{m}_+ := (\Phi^\top \bar{W} \Phi)^{-1}\Phi^\top \bar{W} \expectation{\rv{b}}$ as $M$ tends infinity.

  Next, recall the notation in \eqref{eq: stacked moment function} and consider the function $J(m_+):=\|f(\expectation{\rv{b}};m_+)\|_{\bar{W}}^2$.
  Note that for the ``true" moment vector $\bar{m}_+$, it holds that
  \begin{align*}
    f(\expectation{\rv{b}};\bar{m}_+)= \begin{pmatrix}
      0\\
      (\Phi_{\nu_x}\langle \adj{\trsknl}r -G_2^\top\phi,\mu^N\rangle)_{[1:]}
    \end{pmatrix},
  \end{align*}
  where the upper block vanishes due to the correctly specified moment condition \eqref{eq: moment function},
  and the lower block represents the approximation error from \eqref{eq: moment approximation}. 
  Therefore, it holds that  
  \begin{equation}
  \begin{aligned}
    &\|f(\expectation{\rv{b}};\bar{m}_+)\|_2 =
    \left\| (\Phi_{\nu_x}\langle G_2^\top\phi - \adj{\trsknl}r ,\mu^N\rangle)_{[1:]} \right\|_2\\
    &\le \left\|\Phi_{\nu_x}\langle G_2^\top\phi - \adj{\trsknl}r ,\mu^N\rangle\right\|_2\le \sigma_{\max}(\Phi_{\nu_x})\left\|\langle \adj{\trsknl}r  
                         - G_2^\top\phi,\mu^N\rangle\right\|_2\\
    &\le\! \sigma_{\max}(\Phi_{\nu_x})\!\left\|\adj{\trsknl}r  
                         \!-\! G_2^\top\phi\right\|_{\infty,2}\!\le\! \sigma_{\max}(\Phi_{\nu_x})\!\sqrt{D_V}\left\| \adj{\trsknl}r  
                         \!-\! G_2^\top\phi \right\|_\infty,
  \end{aligned}
  \label{eq: f_inequalities}
  \end{equation}
  where we define the mixed norm $\|f\|_{\infty,2} := \sqrt{\sum_{i=0}^{n-1}(\sup_{\scpair\in\scspace}|(f(\eta))_{[i]}|)^2}$ 
  for vector-valued functions $f:\scspace\rightarrow\mathbb{R}^{n}$,
  and the third inequality results from the fact that $\mu^N$ is a probability measure.

  Since $\tilde{m}_+$ minimizes $J(m_+)$, it follows that
  \begin{align}
    0\le \underbrace{\|f(\expectation{\rv{b}};\tilde{m}_+)\|_{\bar{W}}^2}_{J(\tilde{m}_+)}\le \underbrace{\|f(\expectation{\rv{b}};\bar{m}_+)\|_{\bar{W}}^2}_{J(\bar{m}_+)}. 
    \label{eq: J_inequality_optimality}
  \end{align}
  On the other hand, in view of \eqref{eq: limit of weighted matrix}, it holds that $\sigma_{\min}(\bar{W}) > 0$.
  Also note that $f(\mE[\rv{b}];\tilde{m}_+)-f(\mE[\rv{b}];\bar{m}_+) = \Phi(\bar{m}_+ - \tilde{m}_+)$, hence by \eqref{eq: J_inequality_optimality} and triangular inequality, it holds that
  \begin{align}
    &\sqrt{\sigma_{\min}(\bar{W})}\|\Phi(\tilde{m}_+ - \bar{m}_+)\|_2 \le \|\bar{W}^{1/2} [f(\expectation{\rv{b}};\tilde{m}_+)\!-\!f(\expectation{\rv{b}};\bar{m}_+)]\|_2 \nonumber\\
    &\le \|\bar{W}^{1/2} f(\expectation{\rv{b}};\tilde{m}_+) \|_2 + \|\bar{W}^{1/2} f(\expectation{\rv{b}};\bar{m}_+) \|_2\nonumber\\
    &\overset{\eqref{eq: J_inequality_optimality}}{\le} 2\|\bar{W}^{1/2} f(\expectation{\rv{b}};\bar{m}_+) \|_2\le 2\sqrt{\sigma_{\max}(\bar{W})}\|f(\expectation{\rv{b}};\bar{m}_+)\|_2.\label{eq: phi_delta_m_ineq}
  \end{align}
  Moreover, since $\Phi$ has full column rank, the matrix $\Phi^\top \Phi$ is positive definite, 
  and $\sigma_{\min}(\Phi) = \sqrt{\sigma_{\min}(\Phi^\top \Phi)} > 0$.
  Consequently, $\|\Phi(\bar{m}_+-\tilde{m}_+)\|_2 \ge \sigma_{\min}(\Phi)\|\bar{m}_+-\tilde{m}_+\|_2$.
  Combining this with \eqref{eq: phi_delta_m_ineq} yields
  \begin{align*}
    \|\bar{m}_+ - \tilde{m}_+\|_2 \le 
    \frac{2\sqrt{\sigma_{\max}(\bar{W})}}{\sigma_{\min}(\Phi)\sqrt{\sigma_{\min}(\bar{W})}}
    \|f(\expectation{\rv{b}};\bar{m}_+)\|_2.
    \label{eq: m_tilde_m_bar_inequality}
  \end{align*} 
  Combining the above inequality with \eqref{eq: f_inequalities} yields
  $\|\bar{m}_+-\tilde{m}_+\|_2\le C\left\| \adj{\trsknl}r - G_2^\top\phi \right\|_\infty$,  where
  \begin{align*}
    C := \frac{2\sqrt{\sigma_{\max}(\bar{W})D_V}\,\sigma_{\max}(\Phi_{\nu_x})}
    {\sqrt{\sigma_{\min}(\bar{W})}\,\sigma_{\min}(\Phi)},
  \end{align*}
  which completes the proof.
\end{proof}

The next lemma quantifies how moment estimation errors propagate through the IOC optimization.
Specifically, it evaluates the optimal solutions of both the true problem \eqref{eq: finite dim ioc in expectation} and the estimated problem \eqref{eq: final ioc optimization} 
using the true moment vector $\bar{m}$, and shows that their objective value difference is controlled by the moment estimation error.
\begin{lemma}\label{lem: error propagation}
  The objective value difference of \eqref{eq: ioc_algorithm_obj_fun} evaluated at the optimal solutions of \eqref{eq: finite dim ioc in expectation} 
  and \eqref{eq: final ioc optimization} is controlled by the moment estimation error $\|\bar{m} - \hat{m}^M\|_\infty$.
  More specifically, let $(\optm{\theta}_\psi, \optm{\theta}_\ell, \optm{\theta}_V)$
  and $(\hat{\theta}_\psi^M, \hat{\theta}_\ell^M, \hat{\theta}_V^M)$
  denote the optimal solutions to
  \eqref{eq: finite dim ioc in expectation} and \eqref{eq: final ioc optimization}, respectively,
  and let $\Xi_\psi$ be defined as in \eqref{eq: finite dim ioc-constr-1}.
  Then it holds that
  \begin{align}
    |(\optm{\theta}_\psi - \hat{\theta}_\psi^M)^\top \bar{m}| 
    \le 2\|\bar{m} - \hat{m}^M\|_\infty \|\Xi_\psi\|_1(\beta_V' + \beta_\ell).
  \end{align}
\end{lemma}
\begin{proof}
  For any feasible $\theta_\psi$ that satisfies the constraint \eqref{eq: finite dim ioc-constr-1},
  it follows from the norm bound \eqref{eq: finite dim ioc-constr-4} that
  \begin{align*}
    |(\bar{m} - \hat{m})^\top\theta_\psi| 
    &\le \|\bar{m} - \hat{m}^M\|_\infty \|\theta_\psi\|_1
    \le \|\bar{m} - \hat{m}^M\|_\infty \|\Xi_\psi\|_1(\beta_V' + \beta_\ell).
  \end{align*}
  Since $\optm{\theta}_\psi$ is optimal to \eqref{eq: finite dim ioc in expectation},
  it holds that $ \bar{m}^\top\optm{\theta}_\psi \le \bar{m}^\top\hat{\theta}_\psi^M $.
  Similarly, since $\hat{\theta}_\psi^M$ is optimal to \eqref{eq: final ioc optimization},
  we have $ \hat{m}^{M,\top}\hat{\theta}_\psi^M\le \hat{m}^{M,\top}\optm{\theta}_\psi $.
  Combining these optimality conditions with the bound derived above, it follows that
  \begin{align*}
    \bar{m}^\top\hat{\theta}_\psi^M&\le \hat{m}^{M,\top}\hat{\theta}_\psi^M + (\bar{m} - \hat{m})^\top\hat{\theta}_\psi^M
    \le \hat{m}^{M,\top}\hat{\theta}_\psi^M + |(\bar{m} - \hat{m})^\top\hat{\theta}_\psi^M|\\
    &\le \hat{m}^{M,\top}\hat{\theta}_\psi^M + \|\bar{m} - \hat{m}\|_\infty \|\Xi_\psi\|_1(\beta_V' + \beta_\ell)\\
    &\le \hat{m}^{M,\top}\optm{\theta}_\psi + \|\bar{m} - \hat{m}\|_\infty \|\Xi_\psi\|_1(\beta_V' + \beta_\ell)\\
    &\le \bar{m}^\top\optm{\theta}_\psi + (\hat{m} - \bar{m})^\top\theta_\psi^\star+\|\bar{m} - \hat{m}\|_\infty \|\Xi_\psi\|_1(\beta_V' + \beta_\ell)\\
    &\le \bar{m}^\top\optm{\theta}_\psi + 2\|\bar{m} - \hat{m}\|_\infty \|\Xi_\psi\|_1(\beta_V' + \beta_\ell),
  \end{align*}
  and the statement follows.
\end{proof}

The following lemma addresses the impact of replacing continuous functions (such as $\varphi, \adj{Q}r$) with their polynomial approximations.
It establishes that any feasible solution to the approximated problem can be adjusted to remain feasible in the original problem,
with the objective value difference controlled by the approximation error.
\begin{lemma}\label{lem: basis approximation feasibility}
  Consider the optimization problem
  \begin{equation}\label{eq: lemma2orig}
  \begin{aligned}
    \min_{\theta\in\mathbb{R}^n}\;\; &\langle \theta^\top h , \mu\rangle,\\
    \text{s.t. }\;\; & \theta^\top h \ge 0,
    \int_\scspace \theta^\top h \drm \eta \ge 1,
    \|\theta\|_1 \le \beta,
  \end{aligned}
  \end{equation}
  where $\mu$ is a probability Borel measure on $\scspace$ and
  $h:\scspace\rightarrow\mathbb{R}^n$ is a continuous vector-valued function.
  Suppose $\tilde{h}:\scspace\rightarrow\mathbb{R}^n$ is such that $\|h -\tilde{h}\|_\infty\le 1/2$, 
  and consider the approximated optimization problem
  \begin{equation}\label{eq: lemma2appr}
  \begin{aligned}
    \min_{\theta\in\mathbb{R}^n}\;\; &\langle \theta^\top \tilde{h} , \mu\rangle,\\
    \text{s.t. }\;\; & \theta^\top \tilde{h} \ge 0,
     \int_\scspace \theta^\top \tilde{h}\drm \eta \ge 1,
     \|\theta\|_1 \le \beta/2.
  \end{aligned}
  \end{equation}
  Suppose there exists $\theta_0 \in \mathbb{R}^n$ such that 
  $\theta_0^\top\tilde{h}(\eta) \ge 1$ and $\theta_0^\top h(\eta) \ge 1$ for all $\eta\in\scspace$, 
  and $\|\theta_0\|_1\le 1$.
  Then for any $\theta$ that is feasible to \eqref{eq: lemma2appr},
  let $a := \|\theta\|_1\|h - \tilde{h}\|_\infty$; it holds that $\theta + a\theta_0$ is feasible to \eqref{eq: lemma2orig}.
  Moreover, the corresponding objective values of \eqref{eq: lemma2orig} and \eqref{eq: lemma2appr} satisfy
  \begin{align*}
    \langle(\theta + a\theta_0)^\top h - \theta^\top\tilde{h}, \mu \rangle 
    \le\frac{\beta}{2}\left(1+\|h\|_\infty\right)\|h - \tilde{h}\|_\infty.
  \end{align*}
  Furthermore, denote $\hat{J}$ as the optimal value of \eqref{eq: lemma2appr}. If \eqref{eq: lemma2orig} has an optimal solution $\optm{\theta}$ 
  with $\|\optm{\theta}\|_1 \le \beta/4$, then it holds that 
  \begin{align*}
    0\le \hat{J} \le \langle h^\top\optm{\theta}, \mu\rangle 
    + \frac{\beta}{4}(1+\left\|h\right\|_\infty)\left\|h - \tilde{h}\right\|_\infty.
  \end{align*}
\end{lemma}
\begin{proof}
  \textbf{Part 1 (Feasibility transfer):}
  For any $\theta$ that is feasible to \eqref{eq: lemma2appr}, it follows that
  \begin{align*}
    \theta^\top h = \theta^\top\tilde{h} + \theta^\top(h - \tilde{h})
      \ge \theta^\top\tilde{h} - \underbrace{\|\theta\|_1\|h - \tilde{h}\|_\infty}_{a}.
  \end{align*}
  Then it holds from the above inequality that
  \begin{align*}
    (\theta+a\theta_0)^\top h= \theta^\top h + a\underbrace{\theta_0^\top h}_{\ge 1} \ge \theta^\top h + a\ge \theta^\top\tilde{h}\ge 0,
  \end{align*}
  Furthermore, it follows that $\int (\theta+a\theta_0)^\top h \drm\scpair \ge \int \theta^\top \tilde{h}\ge 1$ and
  \begin{align*}
  \|\theta+a\theta_0\|_1 \le \|\theta\|_1+a\|\theta_0\|_1\le 
    \beta/2 + \underbrace{\|h - \tilde{h}\|_\infty}_{\le 1/2}\underbrace{\|\theta_0\|_1}_{\le 1}\beta/2 \le \beta.
  \end{align*}
  Hence $(\theta+a\theta_0)$ is feasible to \eqref{eq: lemma2orig}.

  \textbf{Part 2 (Objective value proximity):}
  Since $\mu$ is a probability measure, it holds that
  \begin{align*}
    &\langle(\theta + a\theta_0)^\top h - \theta^\top\tilde{h}, \mu \rangle 
    \le \left\|(\theta + a\theta_0)^\top h - \theta^\top\tilde{h}\right\|_\infty\\
    &\le \left\|\theta^\top (h - \tilde{h})\right\|_\infty + 
        \left\|a\theta_0^\top h\right\|_\infty
    \le \|\theta\|_1\left\|h - \tilde{h}\right\|_\infty + 
        a\left\|\theta_0\right\|_1\left\|h\right\|_\infty\\
    &\le \beta\left\|h - \tilde{h}\right\|_\infty/2 + 
        a\underbrace{\|\theta_0\|_1}_{\le 1}\left\|h\right\|_\infty
    \le \frac{\beta}{2}(1+\left\|h\right\|_\infty)\left\|h - \tilde{h}\right\|_\infty.
  \end{align*}
  
  \textbf{Part 3 (Optimal value bound):}
  Conversely, let $a':=\|\optm{\theta}\|_1\|h - \tilde{h}\|_\infty$. Since by assumption $\|\optm{\theta}\|_1 \le \beta/4$, it holds that
  $a' \le (\beta/4)\|h - \tilde{h}\|_\infty$.
  By an argument similar to Part 1, $\optm{\theta} + a'\theta_0$ is feasible to \eqref{eq: lemma2appr} with
  $\|\optm{\theta} + a'\theta_0\|_1 \le \beta/4 + \beta/8 < \beta/2$.

  On the other hand, by following a similar analysis as in Part 2, it holds that
  \begin{align*}
    &\langle(\optm{\theta} + a'\theta_0)^\top \tilde{h} - \theta^{\star\top}h, \mu \rangle 
    \le \frac{\beta}{4}(1+\left\|\tilde{h}\right\|_\infty)\left\|h - \tilde{h}\right\|_\infty.
  \end{align*}
  In view of the optimality of $\hat{J}$, the above inequality further implies that 
  \begin{align*}
    \hat{J}&\le\langle(\optm{\theta} + a'\theta_0)^\top\tilde{h}, \mu\rangle
    \le \langle  h^\top\optm{\theta}, \mu\rangle 
    + \frac{\beta}{4}\left(1+\|\tilde{h}\|_\infty\right)\|h - \tilde{h}\|_\infty.
  \end{align*}
\end{proof}
Combining the above lemmas, the following theorem establishes the main consistency result.
It shows that, under the triple limit of increasing sample size and polynomial approximation degrees,
the objective value achieved by the solution of \eqref{eq: final ioc optimization} converges to the optimal value of the original infinite-dimensional problem \eqref{eq: ioc infinite dim}.
\begin{theorem}
  Let $ (\hat{\rv{\theta}}_\psi, \hat{\rv{\theta}}_\ell, \hat{\rv{\theta}}_V) $ 
  be the optimal solution of \eqref{eq: final ioc optimization}, it holds that
  \begin{align*}
        \lim_{d_V\rightarrow\infty}\lim_{d_\psi\rightarrow\infty}\plim_{M\rightarrow\infty} \;
        \langle \hat{\rv{\theta}}_\ell^\top\varphi + \hat{\rv{\theta}}_V^\top(\alpha\adj{Q}r-r),
         \mu^N \rangle = 0,
  \end{align*}
\end{theorem}
\begin{proof}
  Recall that \eqref{eq: final ioc optimization} and \eqref{eq: finite dim ioc in expectation} have the same feasible set.
  Since $ (\hat{\rv{\theta}}_\psi, \hat{\rv{\theta}}_\ell, \hat{\rv{\theta}}_V) $ is feasible to \eqref{eq: final ioc optimization},
  it is also feasible to \eqref{eq: finite dim ioc in expectation},
  and the latter is an approximation of the following optimization
  \begin{equation}\label{eq: thm-proof-apprx-3}
  \begin{aligned}
      \min_{\psi, \theta_\ell, \theta_V} &\;\;\langle \psi, \mu^N\rangle,\\
      &\;\; \psi = \theta_\ell^\top \varphi + \theta_V^\top (\alpha \adj{\trsknl}r - r),\\
      &\;\; \int_{\scspace} \psi(x,a)\drm x\drm a \ge 1,\\
      &\;\; \psi\in\contfnc{\scspace},\psi(x,a)\ge 0, \forall (x,a)\in\scspace,\\
      &\;\; \|\theta_V\|_1 \le 2\beta_V^\prime, \|\theta_\ell\|_1\le 2\beta_\ell. 
  \end{aligned}
  \end{equation}
  Since in practice, we can choose the norm bounds $\beta_V',\beta_\ell$ for the $\theta_V$ and $\theta_\ell$ to be arbitrarily large such that the optimal solution $(\hat{\psi}_2, \hat{\theta}_{\ell,2}, \hat{\theta}_{V,2})$ for \eqref{eq: thm-proof-apprx-3} satisfies $\|[\hat \theta_{V,2}^\top, \hat \theta_{\ell,2}^\top]^\top\|_1\le (\beta_V'+\beta_\ell)/2$.

  Next, we would like to first analyze the relationship between \eqref{eq: thm-proof-apprx-3} and \eqref{eq: finite dim ioc in expectation} and we want to apply Lemma~\ref{lem: basis approximation feasibility}.
  To this end, let $\hat{\theta} := [\hat{\theta}_\ell^\top, \hat{\theta}_V^\top]^\top$, 
  $h := [\varphi^\top, (\alpha\adj{Q}r-r)^\top]^\top$ and $ \tilde{h} := \Xi_\psi^\top \phi$.
  Note that any constant component in the cost function do not change the optimal control problem.
  Thus, without loss of generality, we assume $h_{[0]} = \tilde{h}_{[0]} = 1$, i.e., there is a ``1'' in the cost function bases, 
  then $\theta_0 = (1, 0, \dots, 0)^\top$.
  Moreover, by Stone-Weierstrass theorem, it holds that $\|h-\tilde{h}\|_\infty\le 1/2$ for large enough $d_\psi$ and $d_V$.
  Hence, by Lemma~\ref{lem: basis approximation feasibility}, it holds that
  \begin{align}
    &| \langle (\hat{\rv{\theta}}+a\theta_0)^\top h - \hat{\rv{\theta}}^\top\tilde{h}, \mu^N \rangle |
    \le (\beta_\ell + \beta_V^\prime)(1\!+\!\|h\|_\infty)\|h-\tilde{h}\|_\infty.\label{eq: theorem-eq-2}
  \end{align}
  Furthermore, recall the notation in Lemma \ref{lem: basis approximation feasibility}, it holds that
  \begin{align}
    &|\langle \hat{\rv{\theta}}_\ell^\top\varphi + \hat{\rv{\theta}}_V^\top(\alpha\adj{Q}r-r) 
        - (\hat{\rv{\theta}} + a\theta_0 )^\top h,
         \mu^N \rangle| = 
         \langle a\theta_0^\top h, \mu^N \rangle \nonumber\\
         &\le |a\underbrace{\theta_0^\top h}_{=1}| = |a|, \label{eq: theorem-eq-1}
  \end{align}
  where $a = \|\hat{\theta}\|_1\|h-\tilde{h}\|_\infty$.
  Moreover, let $[\hat{\theta}_{\psi, 1}^\top, \underbrace{\hat{\theta}_{\ell, 1}^\top, \hat{\theta}_{V, 1}^\top}_{\hat\theta_1^\top}]^\top$ be optimal to
  \eqref{eq: finite dim ioc in expectation}. By Lemma~\ref{lem: error propagation}, it follows that
  \begin{align}
    |  (\hat{\theta}_{\psi,1}-\hat{\rv{\theta}}_\psi)^\top\bar{m}|
    \le 2\|\bar{m} - \hat{m}^M\|_\infty \|\Xi_\psi\|_1(\beta_V' + \beta_\ell).\label{eq: theorem-eq-3}
  \end{align}
  On the other hand, recall that $\langle\phi,\mu^N\rangle = \bar{m}$, $\hat\theta_{\psi,1} = \Xi_\psi\hat{\theta}_1$, $\Xi_\psi^T\phi = \tilde h$, $\hat\psi_2 = \hat\theta_2^\top h$ and apply Lemma \ref{lem: basis approximation feasibility}, it follows that
  \begin{equation}\label{eq: theorem-eq-4}
  \begin{aligned}
    &| \hat{\theta}_{\psi,1}^\top \bar{m} \!-\! \langle \hat{\psi}_2, \mu^N \rangle | 
    =|\langle \hat{\theta}_{\psi,1}^\top\phi \!-\! \hat\psi_2,\mu^N\rangle |
    =|\langle \hat{\theta}_1^\top\tilde{h} \!-\! \hat{\theta}_2^\top h,\mu^N\rangle |\\
    &\le  \frac{\beta_\ell + \beta_V^\prime}{2}(1+\|h\|_\infty)\|h-\tilde{h}\|_\infty.
  \end{aligned}
  \end{equation}
  And similarly, since 
  $
    \hat{\theta}_{\psi}^\top \bar{m}
    = \langle \hat{\theta}_{\psi}^\top\phi, \mu^N \rangle
    =\langle \hat{\theta}^\top\Xi_{\psi}^\top\phi, \mu^N \rangle
    = \langle \hat{\theta}^\top \tilde{h}, \mu^N \rangle,
  $
  by using triangular inequalities repeatedly,
  it further holds that
  \begin{align*}
    &|\langle \hat{\rv{\theta}}_\ell^\top\varphi + \hat{\rv{\theta}}_V^\top(\alpha\adj{Q}r-r) - \hat{\psi}_2, 
         \mu^N \rangle|\\
    &\overset{\eqref{eq: theorem-eq-1}}{\le} |\langle (\hat{\rv{\theta}} + a\theta_0 )^\top h - \hat{\psi}_2, 
         \mu^N \rangle| + (\beta_\ell+\beta_V^\prime)\|h-\tilde{h}\|_\infty\\
    &\overset{\eqref{eq: theorem-eq-2}}{\le} |\langle \hat{\rv{\theta}}^\top\tilde{h} - \hat{\psi}_2, 
         \mu^N \rangle| + (\beta_\ell+\beta_V^\prime)(2+\|h\|_\infty)\|h-\tilde{h}\|_\infty\\
    &= \;\;| \hat{\rv{\theta}}_{\psi}^\top \bar{m} - \langle \hat{\psi}_2,\mu^N \rangle|+ (\beta_\ell+\beta_V^\prime)(2+\|h\|_\infty)\|h-\tilde{h}\|_\infty \\
    &\overset{\eqref{eq: theorem-eq-3}}{\le} 
    | \hat{\theta}_{\psi,1}^\top \bar{m} - \langle \hat{\psi}_2,\mu^N \rangle| 
    + (\beta_\ell+\beta_V^\prime)[2\|\bar{m} - \hat{\rv{m}}^M\|_\infty \|\Xi_\psi\|_1\\
    & \qquad + (2+\|h\|_\infty)\|h-\tilde{h}\|_\infty]\\
    &\overset{\eqref{eq: theorem-eq-4}}{\le} 
      (\beta_\ell+\beta_V^\prime)[ 2\|\bar{m} - \hat{\rv{m}}^M\|_\infty \|\Xi_\psi\|_1+(\frac{5}{2}+\frac{3}{2}\|h\|_\infty)\|h-\tilde{h}\|_\infty].
  \end{align*}
  In addition, by Lemma~\ref{lem: estimator-converge-in-p}, and recall \eqref{eq: proj r G1 phi}, it holds that
  \begin{align*}
    &\plim_{M\rightarrow \infty} \|\hat{\rv{m}}^M-\bar{m}\|_\infty 
    = \|\tilde{m}-\bar{m}\|_\infty \le  C\left\| \adj{\trsknl}r - G_2^\top\phi \right\|_\infty\\
    &=C\left\| \adj{\trsknl}r\! -\!\frac{1}{\alpha}r\!-\! G_2^\top\phi \!+\!\frac{1}{\alpha}r\right\|_\infty \!=\! C\left\| \adj{\trsknl}r \!-\!\frac{1}{\alpha}r\!-\! G_2^\top\phi \!+\!\frac{1}{\alpha}G_1^\top\phi\right\|_\infty\\
    &\le \frac{C}{\alpha}\left\|[(\varphi \!-\! H^\top\phi)^\top,(\alpha \adj{Q}r\!-\!r\!-\!\alpha G_2^\top\phi\!+\!G_1^\top\phi)^\top]^\top\right\|_\infty=\frac{C}{\alpha} \|h \!-\! \tilde{h}\|_\infty
  \end{align*} 
  And by Stone-Weierstrass theorem, as $d_\psi$ tends infinity, $ \|h - \tilde{h}\|_\infty $ converges to zero, hence
  \begin{align*}
    \lim_{d_\psi\rightarrow\infty}\plimsup_{M\rightarrow\infty}
    |\langle \hat{\rv{\theta}}_\ell^\top\varphi + \hat{\rv{\theta}}_V^\top(\alpha\adj{Q}r-r) - \hat{\psi}_2, 
         \mu^N \rangle| = 0
  \end{align*}
  By \citep[Lem.~4.6]{wang2025consistent}, $\langle\hat{\psi}_2, \mu^N\rangle$ shall tend to the optimal value of \eqref{eq: ioc infinite dim} as $d_V\rightarrow\infty$.
  Thus by Proposition~\ref{prop: optimal value zero}, it holds that 
  $\lim_{d_V\rightarrow\infty}\langle\hat{\psi}_2, \mu^N\rangle = 0$ and the conclusion follows.
\end{proof}

\section{Experiments}\label{sec: experiments}

In this section, we evaluate the performance of the proposed algorithm on
(1) linear system and
(2) temperature control system.
In practice, we employ Lagrange interpolant basis and monomial basis as $\phi$ and $r$ respectively.
Moreover, recall Assumption~\ref{Asmp: Compact space}, the state-action space $X\times A$ is Archimedean.
By \citep{putinar1993positive}\citep[Thm.~3.20]{laurent2009sums}, Weighted Sum-of-Square polynomials are dense in non-negative polynomial space on $X\times A$.
Thus, we deal the non-negative constraints \eqref{eq: finite dim ioc-constr-3} with Weighted Sum-of-Square polynomial constraints. 
Then the optimization \eqref{eq: final ioc optimization} is solved in the same way as our prior work \citep{wang2025consistent}.

\subsection{Linear system experiment}\label{sec: linear expe}

We first validate the proposed method on a discrete-time linear system with dynamics
\begin{align*}
  \rv{x}_{t+1} = A\rv{x}_t + B\rv{u}_t + \rv{w}_t,
\end{align*}
where $A = \begin{bmatrix} 1.0 & 0.1 \\ 0 & 1.0 \end{bmatrix}$, 
$B = \begin{bmatrix} 0 \\ 0.1 \end{bmatrix}$,
and $\rv{w}_t \!\sim\! \mathcal{N}_{0.1}(0, 0.01^2 I_2)$ is 
the process noise following truncated normal distribution on the hypercube $[-0.1,0.1]^2$.
The feasible state and action spaces are set to $ [-1,1]^2, [-1,1] $.
The cost function takes the quadratic form
\begin{align*}
  \bar{\ell}(x,u) = x^\top \mathrm{diag}(q_1, q_2) x + r u^2.
\end{align*}
To generate diverse demonstration data, we conduct 100 independent trials,
where in each trial, the cost coefficients $(q_1, q_2, r)$ are sampled uniformly in $[0,1]\times[0,1]\times[0,1]$.
Since the optimal policy remains unchanged under constant scaling of the cost function,
without loss of generality, the coefficient vectors are normalized to unit norm $(q_1^2 + q_2^2 + r^2) = 1$.
The estimated coefficients are similarly normalized, and estimation errors are evaluated based on these normalized vectors.
For each trial, the linear feedback control law is computed via the Algebraic Riccati Equation (ARE) 
with a discount factor $\alpha = 0.9$.
Consequently, $M = 256$ trajectories of length $N = 10$ are generated with initial states drawn 
from truncated normal distribution$\mathcal{N}_{0.9}(0, 0.3^2 I_2)$.
Furthermore, additive Gaussian observation noise with standard deviation $\sigma_{\mathrm{obs}} = 0.05$ 
is applied to both state and action observations.

For this linear system, the polynomial approximation degrees are set as $(d_\psi, d_V) = (2, 2)$.
The GMM estimator follows the procedure in Section~\ref{sec: moment estimator},
with block-diagonal weight matrix \eqref{eq: weighted matrix}, in which we take regularization term $\lambda=10^{-4}$.
The proposed method recovers all three cost coefficients $(\hat{q}_1, \hat{q}_2, \hat{r})$, 
which are then normalized to unit norm for comparison with the ground truth.

Fig.~\ref{fig:linear_error_hist} shows the signed estimation error distributions for $(q_1, q_2)$.
Note that due to the unit-norm constraint $(q_1^2 + q_2^2 + r^2) = 1$, 
the normalized coefficient vector has only two degrees of freedom; 
thus, the errors in $(q_1, q_2)$ fully characterize the estimation accuracy.
As illustrated, the estimation errors concentrate tightly around zero, 
with $(q_1: \text{mean} = 0.0004, \text{std} = 0.0065)$ and $(q_2: \text{mean} = -0.0005, \text{std} = 0.0115)$.
This demonstrates that the proposed IOC method can accurately recover the cost function under noisy observations.

\begin{figure}[!htpb]
  \centering
  \includegraphics[width=\linewidth]{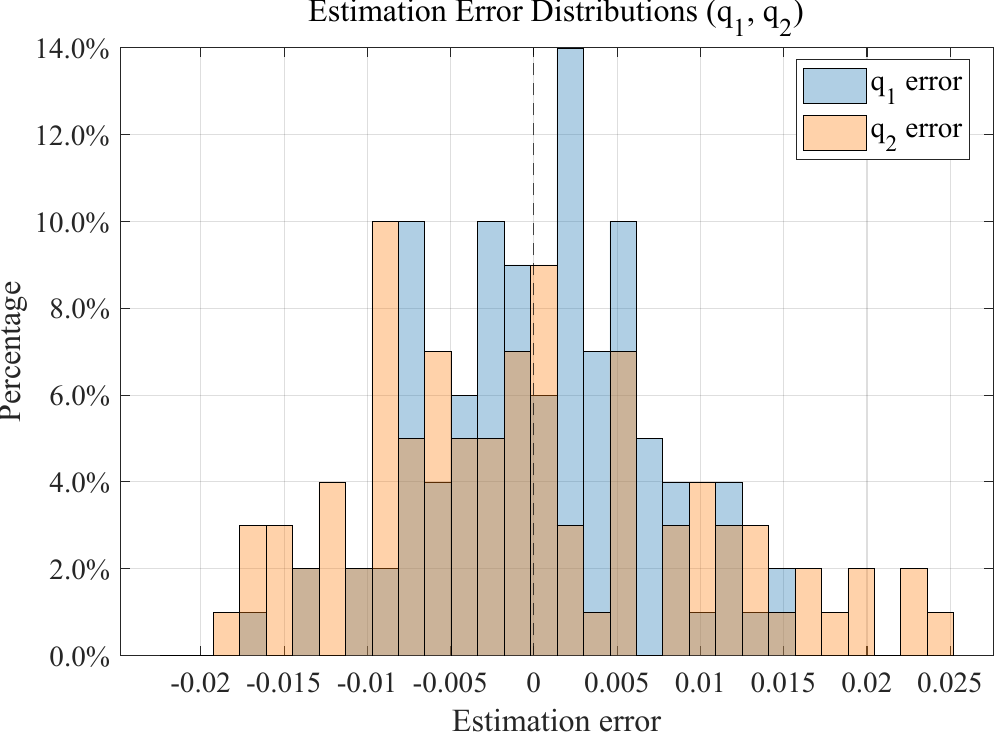}
  \caption{Signed estimation error distributions for $(q_1, q_2)$ over 100 trials.
    Histogram bins span the 1st--99th percentiles (approximately 2\% outliers excluded).
    Both distributions center at zero with small variance, validating the proposed IOC method.}
  \label{fig:linear_error_hist}
\end{figure}

\subsection{Temperature control system experiments}
Next, we evaluate the proposed method on a nonlinear temperature control system.
The system dynamics are governed by heat transfer equations including radiation and convection,
and discretized with a sampling period $\Delta t=1$. More specifically, the temperature evolution is governed by
\begin{align*}
  T_{t+1} = T_{t} + \Delta t/C\left(P_{\text{in}} -  
      h A (T_t - T_{\text{env}}) - \epsilon\sigma_{\text{SB}} A (T^4-T_{env}^4) \right)
\end{align*}
where $T$ is the temperature, $C=500$ is the heat capacity, $P_{\text{in}}$ is the heating power,
$\epsilon=0.9$ is the emissivity, $\sigma_{\text{SB}}=5.67\times 10^{-8}$ is the Stefan-Boltzmann constant, 
$A=0.1$ is the surface area, $h=10$ is the convective heat transfer coefficient, 
and $T_{\text{env}}=293.0$ is the ambient temperature.
The system is normalized to unit scale 
by $x = (T - 300)/100$ and $u = P_{\text{in}}/1000$, and in these coordinates the system takes the form
\begin{align*}
  \rv{x}_{t+1} = \sum_{i=0}^{4} a_i \rv{x}^i_t + b \rv{u}_t + \rv{w}_t,
\end{align*}
where $w_t\sim \mathcal{N}_{0.1}(0, 0.01^2)$ is the process noise, 
and the feasible state and action spaces are both set to $[-1, 1]$.
The cost function has the form
\begin{align*}
  \bar{\ell}(x,u) = q(x - x_{\text{ref}})^2 + r(u - u_{\text{ref}})^2,
\end{align*}
where $x_{\text{ref}} = 0.75$ and $u_{\text{ref}} = -1.0$ are normalized reference values.

For each trial, the cost coefficients $(q, r)$ are sampled uniformly from $(0,1)\times(0,1)$ and normalized to unit norm.
Since the system is nonlinear, we do not have closed-form optimal control laws to generate the demonstration data. Hence a Model Predictive Controller (MPC) with time-horizon length 64 and discount factor $\alpha = 0.9$ 
is used to generate demonstration trajectories.
Each trajectory has length $N = 4$, and the initial states is randomly sampled from normal distribution $\mathcal{N}_{0.9}(0, 0.3^2)$.
The proposed IOC algorithm is employed to estimate the cost coefficients $(\hat q, \hat r)$, which are also normalized to unit norm.
Following the same principle as in Section~\ref{sec: linear expe}, the estimation error is evaluated via the Euclidean distance
$ \| (\hat q, \hat r) - (q_{\text{true}}, r_{\text{true}}) \|_2$ between normalized vectors.

\subsubsection{Impact of observation noise and dataset size}

This experiment investigates how estimation accuracy varies with observation noise level and dataset size.
We conduct 400 independent trials with with three levels of Gaussian observation noise: 
$\sigma_{\text{obs}} \in \{0.01, 0.05, 0.1\}$ ($\mathrm{SNR}\approx 36dB, 22dB, 16dB$).
For each noise level, the dataset size $M$ ranges from 2 to 512 trajectories.
The polynomial basis degrees are chosen as $(d_\psi, d_V) = (6, 2)$.
And in these and following experiments, the regularization term in \eqref{eq: weighted matrix} is set as $\lambda=10^{-4}$.

Fig.~\ref{fig:temp_error_datasize_noise} shows the mean estimation error as a function of dataset size.
With 512 demonstration trajectories,
the mean estimation error across all randomly generated cost coefficients reduces to approximately 0.02 for all three observation noise levels.
This demonstrates the statistical consistency of the proposed IOC algorithm.

\begin{figure}[!htpb]
  \centering
  \includegraphics[width=\linewidth]{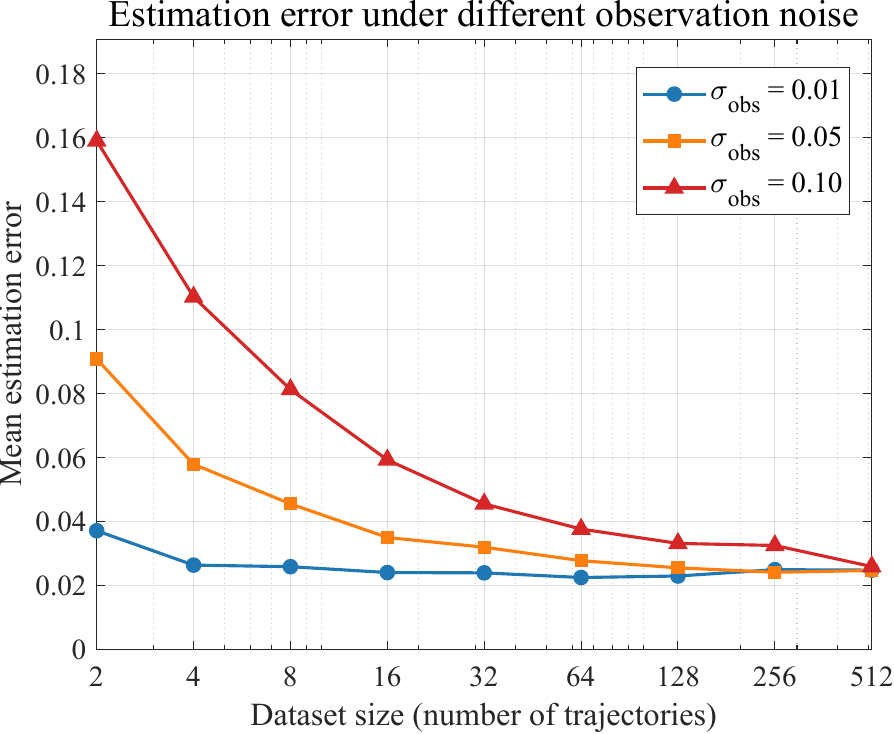}
  \caption{Estimation error vs.\ dataset size under three observation noise levels.
    Lines represent mean errors across 400 trials; higher noise levels produce consistently larger errors.
    As the dataset increases, the estimation error first decreases and then tends to stabilize. }
  \label{fig:temp_error_datasize_noise}
\end{figure}

\subsubsection{Impact of polynomial approximation degree}
As illustrated in Fig.~\ref{fig:temp_error_datasize_noise}, for large enough dataset, 
the estimation error is mainly limited by the approximation accuracy (the estimation error maintains the same level when the dataset size increases).
Thus, this experiment examines the impact of polynomial basis degree on estimation accuracy.
In particular, we fix the observation noise level at $\sigma_{\text{obs}} = 0.05$ 
and increase the dataset size from 4 to 1024 trajectories.
Three polynomial degree configurations are tested: 
$(d_\psi, d_V) = (4, 2)$, $(6, 4)$, and $(10, 4)$, which
represents the increasing approximation capacity for the value function and cost basis.
Each estimation error is averaged over 50 trials.

Fig.~\ref{fig:temp_error_datasize_order} shows 
that higher-degree approximations achieve lower estimation errors
when sufficient data is available.
This validates the theoretical result that the approximation error can be reduced 
by increasing the basis degree, 
provided that sufficient data is available for reliable moment estimation.

\begin{figure}[!htpb]
  \centering
  \includegraphics[width=\linewidth]{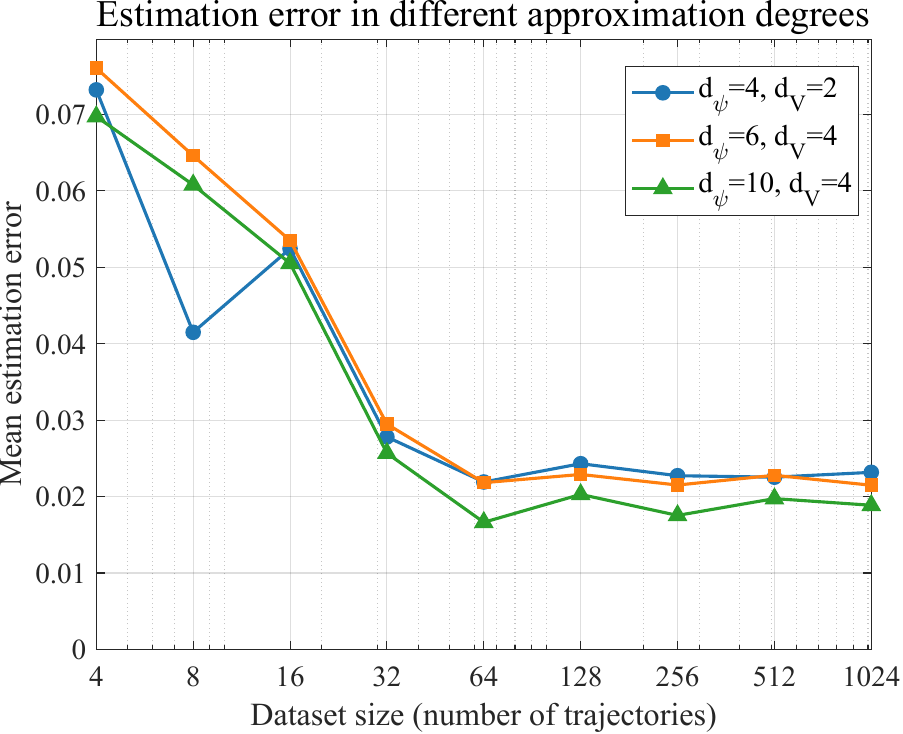}
  \caption{Estimation error vs.\ dataset size for three polynomial basis configurations.
    Higher-degree bases yield better accuracy with large datasets,
    while for small datasets where statistical noise dominates, the performance is similar.}
  \label{fig:temp_error_datasize_order}
\end{figure}

\section{Conclusion}
In this work, we proposed an IOC algorithm for discrete-time infinite horizon nonlinear systems under noisy observations.
In particular, the transition kernel is only assumed to be weak Feller which includes both deterministic and stochastic systems.
The proposed method is based on the necessary and sufficient optimality conditions induced by the occupation measure formulation. The resulting infinite dimensional optimization problem is approximated via finite-dimensional polynomial bases and the corresponding moments of the occupation measure are estimated via a GMM estimator.
We established the consistency of the proposed IOC algorithm, i.e., the estimated cost function converges to the true cost function as the number of demonstration trajectories and the polynomial basis degrees tend infinity. This guarantees that the proposed method can recover the true cost function with arbitrary accuracy given sufficient data and approximation capacity. Moreover, since the resulting IOC optimization is convex, it also spares itself from local minima issues and make sure the theoretical consistency can actually be attained in practice. Through numerical experiments on linear and nonlinear systems, we demonstrate the efficacy of the proposed IOC algorithm.




\bibliographystyle{elsarticle-num} 
\bibliography{ref.bib}

@article{hall2003large,
  title={The large sample behaviour of the generalized method of moments estimator in misspecified models},
  author={Hall, Alastair R and Inoue, Atsushi},
  journal={Journal of Econometrics},
  volume={114},
  number={2},
  pages={361--394},
  year={2003},
  publisher={Elsevier}
}

@book{bertsekas1996stochastic,
  title={Stochastic optimal control: the discrete-time case},
  author={Bertsekas, Dimitri and Shreve, Steven E},
  volume={5},
  year={1996},
  publisher={Athena Scientific}
}

@article{wang2025consistent,
  title={Consistent inverse optimal control for discrete-time nonlinear stochastic systems},
  author={Wang, Ziliang and Zhang, Han and Ringh, Axel},
  journal={arXiv preprint arXiv:2511.22579},
  year={2025}
}

@incollection{laurent2008sums,
  title={Sums of squares, moment matrices and optimization over polynomials},
  author={Laurent, Monique},
  booktitle={Emerging applications of algebraic geometry},
  pages={157--270},
  year={2008},
  publisher={Springer}
}

@Article{harris2020array,
 title         = {Array programming with {NumPy}},
 author        = {Charles R. Harris and K. Jarrod Millman and St{\'{e}}fan J.
                 van der Walt and Ralf Gommers and Pauli Virtanen and David
                 Cournapeau and Eric Wieser and Julian Taylor and Sebastian
                 Berg and Nathaniel J. Smith and Robert Kern and Matti Picus
                 and Stephan Hoyer and Marten H. van Kerkwijk and Matthew
                 Brett and Allan Haldane and Jaime Fern{\'{a}}ndez del
                 R{\'{i}}o and Mark Wiebe and Pearu Peterson and Pierre
                 G{\'{e}}rard-Marchant and Kevin Sheppard and Tyler Reddy and
                 Warren Weckesser and Hameer Abbasi and Christoph Gohlke and
                 Travis E. Oliphant},
 year          = {2020},
 month         = sep,
 journal       = {Nature},
 volume        = {585},
 number        = {7825},
 pages         = {357--362},
 doi           = {10.1038/s41586-020-2649-2},
 publisher     = {Springer Science and Business Media {LLC}},
 url           = {https://doi.org/10.1038/s41586-020-2649-2}
}

@article{likmeta2021dealing,
  title={Dealing with multiple experts and non-stationarity in inverse reinforcement learning: an application to real-life problems},
  author={Likmeta, Amarildo and Metelli, Alberto Maria and Ramponi, Giorgia and Tirinzoni, Andrea and Giuliani, Matteo and Restelli, Marcello},
  journal={Machine Learning},
  volume={110},
  number={9},
  pages={2541--2576},
  year={2021},
  publisher={Springer}
}

@inproceedings{abbeel2004apprenticeship,
  title={Apprenticeship learning via inverse reinforcement learning},
  author={Abbeel, Pieter and Ng, Andrew Y},
  booktitle={Proceedings of the twenty-first international conference on Machine learning},
  pages={1},
  year={2004}
}

@article{ho2016generative,
  title={Generative adversarial imitation learning},
  author={Ho, Jonathan and Ermon, Stefano},
  journal={Advances in neural information processing systems},
  volume={29},
  year={2016}
}

@INPROCEEDINGS{Menner2018Convex,
  author={Menner, Marcel and Zeilinger, Melanie N.},
  booktitle={2018 European Control Conference (ECC)}, 
  title={Convex Formulations and Algebraic Solutions for Linear Quadratic Inverse Optimal Control Problems}, 
  year={2018},
  volume={},
  number={},
  pages={2107-2112},
  doi={10.23919/ECC.2018.8550090}
}

@article{fernandez2025estimating,
  title={Estimating unknown dynamics and cost as a bilinear system with Koopman-based Inverse Optimal Control},
  author={Fernandez-Ayala, Victor Nan and Deka, Shankar A and Dimarogonas, Dimos V},
  journal={arXiv preprint arXiv:2501.18318},
  year={2025}
}

@article{lian2022inverse,
  title={Inverse reinforcement learning for multi-player noncooperative apprentice games},
  author={Lian, Bosen and Xue, Wenqian and Lewis, Frank L and Chai, Tianyou},
  journal={Automatica},
  volume={145},
  pages={110524},
  year={2022},
  publisher={Elsevier}
}

@INPROCEEDINGS{Rodrigues2022Inverse,
  author={Rodrigues, Luis},
  booktitle={2022 IEEE 61st Conference on Decision and Control (CDC)}, 
  title={Inverse Optimal Control with Discount Factor for Continuous and Discrete-Time Control-Affine Systems and Reinforcement Learning}, 
  year={2022},
  volume={},
  number={},
  pages={5783-5788},
  doi={10.1109/CDC51059.2022.9992796}}

@inproceedings{ziebart2008maximum,
  title={Maximum entropy inverse reinforcement learning.},
  author={Ziebart, Brian D and Maas, Andrew L and Bagnell, J Andrew and Dey, Anind K and others},
  booktitle={Aaai},
  volume={8},
  pages={1433--1438},
  year={2008},
  organization={Chicago, IL, USA}
}

@ARTICLE{zhou2018Infinite,
  author={Zhou, Zhengyuan and Bloem, Michael and Bambos, Nicholas},
  journal={IEEE Transactions on Automatic Control}, 
  title={Infinite Time Horizon Maximum Causal Entropy Inverse Reinforcement Learning}, 
  year={2018},
  volume={63},
  number={9},
  pages={2787-2802},
  doi={10.1109/TAC.2017.2775960}
}

@ARTICLE{Mehr2023Maximum,
  author={Mehr, Negar and Wang, Mingyu and Bhatt, Maulik and Schwager, Mac},
  journal={IEEE Transactions on Robotics}, 
  title={Maximum-Entropy Multi-Agent Dynamic Games: Forward and Inverse Solutions}, 
  year={2023},
  volume={39},
  number={3},
  pages={1801-1815},
  doi={10.1109/TRO.2022.3232300}
}

@article{garrabe2025convex,
  title={On convex data-driven inverse optimal control for nonlinear, non-stationary and stochastic systems},
  author={Garrabe, Emiland and Jesawada, Hozefa and Del Vecchio, Carmen and Russo, Giovanni},
  journal={Automatica},
  volume={173},
  pages={112015},
  year={2025},
  publisher={Elsevier}
}

@article{zhang2019inverse,
  title={Inverse optimal control for discrete-time finite-horizon Linear Quadratic Regulators},
  author={Zhang, Han and Umenberger, Jack and Hu, Xiaoming},
  journal={Automatica},
  volume={110},
  pages={108593},
  year={2019},
  publisher={Elsevier}
}

@inproceedings{zhang2019inverseCDC,
  title={Inverse optimal control for finite-horizon discrete-time linear quadratic regulator under noisy output},
  author={Zhang, Han and Li, Yibei and Hu, Xiaoming},
  booktitle={2019 IEEE 58th Conference on Decision and Control (CDC)},
  pages={6663--6668},
  year={2019},
  organization={IEEE}
}

@inproceedings{li2018convex,
  title={A convex optimization approach to inverse optimal control},
  author={Li, Yibei and Zhang, Han and Yao, Yu and Hu, Xiaoming},
  booktitle={2018 37th Chinese Control Conference (CCC)},
  pages={257--262},
  year={2018},
  organization={IEEE}
}

@article{yu2021system,
  title={System identification approach for inverse optimal control of finite-horizon linear quadratic regulators},
  author={Yu, Chengpu and Li, Yao and Fang, Hao and Chen, Jie},
  journal={Automatica},
  volume={129},
  pages={109636},
  year={2021},
  publisher={Elsevier}
}

@article{mombaur2010human,
  title={From human to humanoid locomotion -- an inverse optimal control approach},
  author={Mombaur, Katja and Truong, Anh and Laumond, Jean-Paul},
  journal={Autonomous robots},
  volume={28},
  number={3},
  pages={369--383},
  year={2010},
  publisher={Springer}
}

@article{jin2019inverse,
  title={Inverse optimal control for multiphase cost functions},
  author={Jin, Wanxin and Kuli{\'c}, Dana and Lin, Jonathan Feng-Shun and Mou, Shaoshuai and Hirche, Sandra},
  journal={IEEE Transactions on Robotics},
  volume={35},
  number={6},
  pages={1387--1398},
  year={2019},
  publisher={IEEE}
}

@inproceedings{keshavarz2011imputing,
  title={Imputing a convex objective function},
  author={Keshavarz, Arezou and Wang, Yang and Boyd, Stephen},
  booktitle={2011 IEEE international symposium on intelligent control},
  pages={613--619},
  year={2011},
  organization={IEEE}
}

@article{molloy2018finite,
  title={Finite-horizon inverse optimal control for discrete-time nonlinear systems},
  author={Molloy, Timothy L and Ford, Jason J and Perez, Tristan},
  journal={Automatica},
  volume={87},
  pages={442--446},
  year={2018},
  publisher={Elsevier}
}

@article{aswani2018inverse,
  title={Inverse optimization with noisy data},
  author={Aswani, Anil and Shen, Zuo-Jun and Siddiq, Auyon},
  journal={Operations Research},
  volume={66},
  number={3},
  pages={870--892},
  year={2018},
  publisher={INFORMS}
}

@article{molloy2020online,
  title={Online inverse optimal control for control-constrained discrete-time systems on finite and infinite horizons},
  author={Molloy, Timothy L and Ford, Jason J and Perez, Tristan},
  journal={Automatica},
  volume={120},
  pages={109109},
  year={2020},
  publisher={Elsevier}
}

@article{pauwels2016linear,
  title={Linear conic optimization for inverse optimal control},
  author={Pauwels, Edouard and Henrion, Didier and Lasserre, Jean-Bernard},
  journal={SIAM Journal on Control and Optimization},
  volume={54},
  number={3},
  pages={1798--1825},
  year={2016},
  publisher={SIAM}
}

@inproceedings{rouot2017inverse,
  title={On inverse optimal control via polynomial optimization},
  author={Rouot, J{\'e}r{\'e}my and Lasserre, Jean-Bernard},
  booktitle={2017 IEEE 56th Annual Conference on Decision and Control (CDC)},
  pages={721--726},
  year={2017},
  organization={IEEE}
}

@article{zhang2021inverse,
  title={Inverse linear-quadratic discrete-time finite-horizon optimal control for indistinguishable homogeneous agents: A convex optimization approach},
  author={Zhang, Han and Ringh, Axel},
  journal={Automatica},
  volume={148},
  pages={110758},
  year={2023},
  publisher={Elsevier}
}

@inproceedings{zhang2022statistically,
  author={Zhang, Han and Ringh, Axel and Jiang, Weihan and Li, Shaoyuan and Hu, Xiaoming},
  booktitle={2022 41st Chinese Control Conference (CCC)}, 
  title={Statistically Consistent Inverse Optimal Control for Linear-Quadratic Tracking with Random Time Horizon}, 
  year={2022},
  volume={},
  number={},
  pages={1515-1522},
  doi={10.23919/CCC55666.2022.9902327}
}

@book{hernandez2012discrete,
  title={Discrete-time Markov control processes: basic optimality criteria},
  author={Hern{\'a}ndez-Lerma, On{\'e}simo and Lasserre, Jean Bernard},
  year={1996},
  publisher={Springer Science \& Business Media}
}

@article{putinar1993positive,
  title={Positive polynomials on compact semi-algebraic sets},
  author={Putinar, Mihai},
  journal={Indiana University Mathematics Journal},
  volume={42},
  number={3},
  pages={969--984},
  year={1993},
  publisher={JSTOR}
}

@incollection{laurent2009sums,
  title={Sums of squares, moment matrices and optimization over polynomials},
  author={Laurent, Monique},
  booktitle={Emerging applications of algebraic geometry},
  pages={157--270},
  year={2009},
  publisher={Springer}
}

@inproceedings{rickenbach2024inverse,
  title={Inverse optimal control as an errors-in-variables problem},
  author={Rickenbach, Rahel and Scampicchio, Anna and Zeilinger, Melanie N},
  booktitle={6th Annual Learning for Dynamics \& Control Conference},
  pages={375--386},
  year={2024},
  organization={PMLR}
}

@article{guo2023imitation,
  title={Imitation and transfer learning for LQG control},
  author={Guo, Taosha and Al Makdah, Abed AlRahman and Krishnan, Vishaal and Pasqualetti, Fabio},
  journal={IEEE Control Systems Letters},
  volume={7},
  pages={2149--2154},
  year={2023},
  publisher={IEEE}
}

@article{zhang2024statistically,
  title={Statistically consistent inverse optimal control for discrete-time indefinite linear--quadratic systems},
  author={Zhang, Han and Ringh, Axel},
  journal={Automatica},
  volume={166},
  pages={111705},
  year={2024},
  publisher={Elsevier}
}

@article{zhang2024inverse,
  title={Inverse optimal control for averaged cost per stage linear quadratic regulators},
  author={Zhang, Han and Ringh, Axel},
  journal={Systems \& Control Letters},
  volume={183},
  pages={105658},
  year={2024},
  publisher={Elsevier}
}

@article{li2026inverse,
  title={Inverse optimal control for linear quadratic tracking with unknown target states},
  author={Li, Yao and Yu, Chengpu and Fang, Hao and Chen, Jie},
  journal={Automatica},
  volume={185},
  pages={112819},
  year={2026},
  publisher={Elsevier}
}



\end{document}